\documentclass[a4paper,12pt]{article}
\usepackage{amsmath,amsthm,amsfonts,amssymb,bbm}
\usepackage{graphicx,psfrag,subfigure,color,cite}
\usepackage[UKenglish]{babel}
\usepackage{textcmds}
\usepackage{mathtools}
\usepackage{tabularx}
\usepackage{float}
\usepackage{subcaption}
\usepackage{changes}
\usepackage[colorlinks]{hyperref}

\usepackage{epstopdf}
\numberwithin{equation}{section}

\usepackage[margin=1.2in]{geometry}

\newcommand{\e}{\varepsilon}
\newcommand{\Pb}{\mathbb{P}}

\newcommand{\E}{\mathbb{E}}

\newcommand{\R}{\mathbb{R}}

\newcommand{\Z}{\mathbb{Z}}

\newcommand{\Id}{\mathbbm{1}}
\newcommand{\Ai}{\mathrm{Ai}}
\renewcommand{\R}{\mathbb{R}}

\newtheorem{prop}{Proposition}[section]

\newtheorem{thm}[prop]{Theorem}
\newtheorem{lem}[prop]{Lemma}
\newtheorem{defin}[prop]{Definition}

\newtheorem{conj}[prop]{Conjecture}
\newtheorem{cla}[prop]{Claim}
\newtheorem{asmp}[prop]{Assumption}

\newtheorem{rem}[prop]{Remark}

\newcommand{\newasep}{\textrm{ASEPsc}}

\DeclareMathOperator*{\argmax}{argmax}
\DeclareMathOperator*{\sargmax}{sargmax}

\title{Quasi-geodesics in integrable and non-integrable exclusion processes}

\author{Patrik L.\ Ferrari\thanks{Institute for Applied Mathematics, Bonn University, Endenicher Allee 60, 53115 Bonn, Germany. E-mail: {\tt ferrari@uni-bonn.de}} \and
	Min Liu\thanks{Institute for Applied Mathematics, Bonn University, Endenicher Allee 60, 53115 Bonn, Germany. E-mail: {\tt mliu@uni-bonn.de}}
}

\date{\today}

\begin{document}
	\maketitle
	\sloppy
	
	\begin{abstract}
	Backwards geodesics for TASEP were introduced in \cite{Fer18}. We consider flat initial conditions and show that under proper scaling its end-point converges to maximizer argument of the Airy$_2$ process minus a parabola. We generalize its definition to generic non-integrable models including ASEP and speed changed ASEP (call it quasi-geodesics). We numerically verify that its end-point is universal, where the scaling coefficients are analytically computed through the KPZ scaling theory.
	\end{abstract}

\section{Introduction}
	
We consider interacting particle systems in the Kardar-Parisi-Zhang (KPZ) universality class, with particle configurations $\eta\in \{0,1\}^\Z$, where $\eta(x)=1$ whenever there is a particle at site $x$. A particle at site $x$ jumps to a free site $y$ with jump rate given by a local function $c_\eta(x,y)$. In this paper we consider three variants of exclusion processes with nearest-neighbor jumps, namely the totally asymmetric simple exclusion process (TASEP), the partially asymmetric simple exclusion process (ASEP), and the speed changed ASEP (ASEPsc). In the latter the function $c_\eta(x,x+1)$ and $c_\eta(x+1,x)$ depend on $\eta(x-1)$ and $\eta(x+2)$ as well.
While TASEP can be analyzed with exact formulas, and partially it is also the case for ASEP, the speed changed ASEP considered here is non-integrable with known stationary measures~\cite{KS92}, allowing us to analytically compute all non-universal coefficients for the KPZ scaling theory~\cite{Spo14,PS01}.

Since jumps are nearest-neighbor, the particle ordering is maintained. We label particles from right to left and call $X_n(t)$ the position of particle $n$ at time $t$. One way to see this model as a growing interface is to consider $n\mapsto X_n(t)+n$ as height function. This is a rotation by 45 degrees of the (more standard) height function whose gradient is given by $1-2\eta(x)$.

As the KPZ universality class has $1:2:3$ scaling, fluctuations of the interface grow as $t^{1/3}$ and space-time correlations are non-trivial in a $t^{2/3}$-neighborhood of characteristic lines. In analogy to last passage percolation (LPP) models (and its universal scaling limit, the directed landscape ${\cal L}$~\cite{DOV22}), one can define a notion of geodesics: for a given $(N,t)$, any trajectory $(X_{N(\tau)}(\tau))_{\tau:t\to 0}$ satisfying the concatenation property
\begin{equation}\label{eq1}
X_N(t)=X_{N(\tau)}(\tau)+X^{{\rm step}}_{N-{N(\tau)}}(\tau,t),
\end{equation}
where $X^{{\rm step}}_{N-{N(\tau)}}(\tau,t)$ is the position of the particle number $N-{N(\tau)}$ of TASEP which starts at time $\tau$ with particles occupying all sites to the left of $X_{N(\tau)}(\tau)$. The reason we call it (backwards) geodesics, is in~\cite{Sep98c} it was shown that
\begin{equation}\label{eq2}
X_N(t)=\min_{M\leq N}\left\{X_{M}(\tau)+X^{{\rm step}}_{N-M}(\tau,t)\right\},
\end{equation}
so along these paths the minimization \eqref{eq2} is fulfilled.

Geodesics have attracted a lot of interest and many properties have been studied, mostly in the framework of LPP (see e.g.~\cite{BF20,BSS19,Zha20,Pim16,BBS20,BGHH20,BBS20b,BHS18,JRS23}), but also in the scaling limit of the directed landscape (see e.g.~\cite{Bus24,BGH22,Dau23,DSV22,RV23}). Properties of geodesics, in particular of localization, have also been very important to study other observables, like the decay of the covariance in the Airy$_1$ process~\cite{BBF22}, the universality of first order correction of the time-time covariance~\cite{FO18,FO22,BG18,BGZ19}, the analysis of mixing times~\cite{SS22}, but also the height function in presence of shocks without passing through maps to LPP~\cite{BF22,FN19,BBF21,FG24,FN23}.

In~\cite{Fer18} a dynamical construction of a backwards geodesic was proposed. In short, the geodesic follows backwards the trajectory of a particle (starting with particle $N$ at time $t$) and every time the particle on the geodesics is prevented from jumping, it follows the backwards trajectory of the blocking particle. This construction is natural since it keeps track of which regions in space-time actually influenced the position of particle $N$ at time $t$, $X_N(t)$. In particular, if we know that the backwards geodesics is in a given space-time non-random region $D$ (with high probability), then the position $X_N(t)$ is independent of the randomness outside the region $D$ (see Lemma~3.1 of~\cite{FG24} for an explicit statement).

The main goal of this work is to investigate whether the generalization of the dynamical construction of backwards geodesics applied to non- or partially-integrable systems such as speed changed ASEP or ASEP is universal. For these models \eqref{eq1} is no longer satisfied, but as proven for ASEP~\cite{QS20} the height functions converge in the scaling limit to the KPZ fixed point, for which
\begin{equation}
h(x,t)=\max_{y\in\R} \{h(y,0)+{\cal L}(y,0;x,t)\}.
\end{equation}
By universality we expect the same to be true for speed changed ASEP and we call our physically motivated generalization \emph{quasi-geodesics}.

For our study we consider flat initial condition (with density $1/2$), that is, $X_N(0)=-2N$ for $N\in\Z$. First of all, we prove for TASEP in Theorem~\ref{uhat} that the end-point of the backwards geodesics, $x_{N(0)}(0)$, in the $t^{2/3}$ scale, has a limit law given by the $\argmax_{u\in\R}\{{\cal A}_2(u)-u^2\}$, where ${\cal A}_2$ is the Airy$_2$ process~\cite{PS02}. The analogue result for line-to-point exponential LPP can be found in~\cite{SCH12,BKS12}.

To verify the universality conjecture, we first need to derive precise conjectures based on the scaling from the KPZ scaling theory, see Section~\ref{kpz}: all the non-universal model-dependent parameters are computed analytically allowing us to formulate precise conjectures for the end-point of the quasi-geodesics without free parameters, see Conjecture~\ref{asepconj} for ASEP and Conjecture~\ref{conj} for speed changed ASEP. In Section~\ref{sectSimulation} we provide numerical evidence that these conjectures hold true.

For the speed changed ASEP, we also verify that with flat initial condition the one-point distribution is asymptotically given by the GOE Tracy-Widom distribution as expected by universality, see Section~\ref{asepscpp}.

Finally, as mentioned above, for ASEP it is known that \eqref{eq2} is only an inequality (see e.g.~\cite{BF22,FN19}). From the results of~\cite{QS20} it follows that the difference
\begin{equation}\label{eq3}
\min_{M\leq N}\left\{X_{M}(\tau)+X^{{\rm step}}_{N-M}(\tau,t)\right\}-X_N(t) \geq 0,
\end{equation}
once divided by $t^{1/3}$, goes to zero as $t\to\infty$ (see Lemma~\ref{lemScaling}). However the speed of convergence to zero remains unknown. For that reason, using our quasi-geodesic, we numerically study the discrepancy
\begin{equation}
d(t)=X_{N(\tau)}(\tau)+X^{{\rm step}}_{N-{N(\tau)}}(\tau,t) - X_N(t)\geq 0,
\end{equation}
which clearly is a bound for \eqref{eq3}. Surprisingly the numerical simulation indicates that $d(t)$ tends to a finite random variable, without the need to be divided by $t^\theta$ for some $\theta\in (0,1/3)$, see Section~\ref{asepdis}. Thus for ASEP the correction to \eqref{eq2} is of order one. In a recent paper~\cite{ACH24} it is shown that the height function has bounded discrepancy from the maximum of some LPP line ensembles (see Section~4 of~\cite{ACH24}). However, it remains unclear how to relate the entries in the line ensembles with ASEP height functions.

\paragraph{Outline.} We present our results in Section~\ref{sec2}. In Section~\ref{sec21}, we introduce the models ASEP and speed changed ASEP. The generalized backward geodesic is defined in Section~\ref{sec22}. Analytical results for the backward geodesic in TASEP are given in Section~\ref{sec23} and proved later in Section~\ref{pl22}. The exact scaling formulas for non-integrable models are also provided in Section~\ref{sec23}, with a detailed explanation on how to use KPZ scaling theory to derive these results given in Section~\ref{kpz}. Finally, in Section~\ref{sectSimulation}, we present numerical results to verify the conjectures mentioned in previous sections.
	
\paragraph{Acknowledgements:} The work was partly funded by the Deutsche Forschungsgemeinschaft (DFG, German Research Foundation) under Germany’s Excellence Strategy - GZ 2047/1, projekt-id 390685813 and by the Deutsche Forschungsgemeinschaft (DFG, German Research Foundation) - Projektnummer 211504053 - SFB 1060. We authors are grateful to Herbert Spohn for indicating the non-integrable model with computable non-universal scaling coefficients, to Dominik Schmid and Ofer Busani for discussions, as well as Milind Hedge for explaining their approximate LPP representation.

	\section{Model and results}\label{sec2}
	\subsection{Simple exclusion process}\label{sec21}
	The simple exclusion process ($\mathrm{SEP}$) is among the interacting particle systems introduced by Spitzer \cite{Spi70}. It is a Markov process $\eta_t$ on configuration space $\Omega=\{0,1\}^\Z$ that describes the motion of particles on $\Z$, where $\eta_t(i)=1$ (resp.\ $\eta_t(i)=0$) stands for presence (resp.\ absence) of a particle at position $i$ at time $t$. For a configuration $\eta\in\Omega$ and $x,y\in\Z$, we define $\eta^{x,y}\in\Omega$ as
	\begin{equation}
		\eta^{x,y}(i)=\begin{cases}
			\eta(y),\quad&\mathrm{if}\ i=x,\\
			\eta(x),&\mathrm{if}\ i=y,\\
			\eta(i),&\mathrm{otherwise}.
		\end{cases}
	\end{equation}
	The generator of $\mathrm{SEP}$ is given by
	\begin{equation}\label{jumpra}
		\mathcal Lf(\eta)=\sum_{\begin{subarray}{c}x,y\in\Z\\|x-y|=1\end{subarray}}c_\eta(x,y)\eta(x)(1-\eta(y))(f(\eta^{x,y})-f(\eta))
	\end{equation}
	for any cylinder function $f:\Omega\to\R$ and where $c_\eta(x,y)$ is the jump rate from site $x$ to site $y$. In this work, we will consider three special cases:
	
	\emph{(a) TASEP:} The totally asymmetric simple exclusion process (TASEP) has jump rates given by
	\begin{equation}
		c_\eta(x,y)=\begin{cases}
			1,\quad &\mathrm{if}\ y=x+1,\\
			0,&\mathrm{otherwise}.
		\end{cases}
	\end{equation}

	\emph{(b) ASEP:} The asymmetric simple exclusion process (ASEP) with asymmetry parameter $p\in (1/2,1]$ is defined by the jump rate in \eqref{jumpra}
	\begin{equation}\label{aseprate}
		c_\eta(x,y)=\begin{cases}
			p,\quad &\mathrm{if}\ y=x+1,\\
			1-p,&\mathrm{if}\ y=x-1,\\
			0,&\mathrm{otherwise},
		\end{cases}
	\end{equation}
	see Figure~\ref{asepjump} for an illustration.
	
	\begin{figure}[h]
		\centering
		\includegraphics{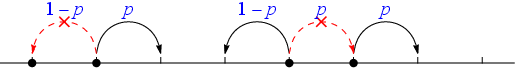}
		\caption{Illustration for jump rate in ASEP. In the setting of ASEP, every particle will attempt to jump its nearest neighbour independently with jump rate $p$ to the right and $1-p$ to the left. The attempted jump is successful only if the target site is unoccupied. }\label{asepjump}
	\end{figure}

	\emph{(c) Speed changed ASEP:} Speed changed ASEP (\newasep) are models where the jump rate is not constant, but it depends on the local particle configuration. We consider here a model with next-nearest-neighbour interactions as in~\cite{KS92}. It is still in the family of SEP since the jumps are nearest-neighbour, that is, for $|i|\geq 2$, $c_\eta(x,x+i)=0$. The jump rates from $x$ to $x+1$ and $x-1$ are given as follows. Let $\alpha_i,\gamma_i\geq 0$ with $i\in\{2,3,4\}$ be fixed parameters. Then the right jumps occur with rate
	\begin{equation}\label{jp1}
		\begin{aligned}
			c_\eta(j,j+1)=\begin{cases}
				\alpha_2,\quad&\mathrm{if}\ \eta(j-1)=0,\eta(j)=1,\eta(j+1)=0,\eta(j+2)=1,\\
				\alpha_3,\quad&\mathrm{if}\ \eta(j-1)=0,\eta(j)=1,\eta(j+1)=0,\eta(j+2)=0,\\
				\alpha_3,\quad&\mathrm{if}\ \eta(j-1)=1,\eta(j)=1,\eta(j+1)=0,\eta(j+2)=1,\\
				\alpha_4,\quad&\mathrm{if}\ \eta(j-1)=1,\eta(j)=1,\eta(j+1)=0,\eta(j+2)=0,\\
			\end{cases}
		\end{aligned}
	\end{equation}
	while the left jumps have rates
	\begin{equation}\label{jp2}
		\begin{aligned}
			c_\eta(j+1,j)=\begin{cases}
				\gamma_2,\quad&\mathrm{if}\ \eta(j-1)=0,\eta(j)=0,\eta(j+1)=1,\eta(j+2)=1,\\
				\gamma_3,\quad&\mathrm{if}\ \eta(j-1)=0,\eta(j)=0,\eta(j+1)=1,\eta(j+2)=0,\\
				\gamma_3,\quad&\mathrm{if}\ \eta(j-1)=1,\eta(j)=0,\eta(j+1)=1,\eta(j+2)=1,\\
				\gamma_4,\quad&\mathrm{if}\ \eta(j-1)=1,\eta(j)=0,\eta(j+1)=1,\eta(j+2)=0,\\
			\end{cases}
		\end{aligned}
	\end{equation}
	see Figure~\ref{jumprate} for an illustration.
	\begin{figure}[h]
		\centering
		\includegraphics{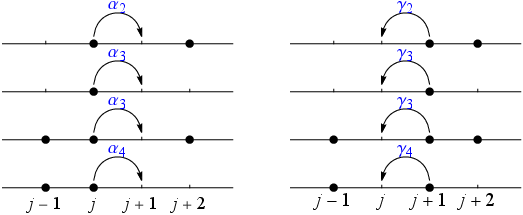}
		\caption{Illustration for jump rate of process \newasep. The left figure is for right jump rates \eqref{jp1}, while the right figure is for left jump rates \eqref{jp2}.}\label{jumprate}
	\end{figure}
	
	For TASEP and ASEP many analytic results have been obtained, many of them due to the presence of some integrable structure. However not every observable one is easily accessible by exact formulas. The reason why we have chosen \newasep, is that, on the one hand it does not have an integrable structure as in (T)ASEP allowing an exact asymptotic analysis, on the other hand the stationary measure is known and with it we can compute all the model-dependent coefficients which enters in the KPZ scaling theory. As a consequence, we do not have any free parameter to be numerically fitted in order to verify the conjectures.
	
	Due to the exclusion principle and nearest-neighbour jumps, the order of the particles remains unchanged. Therefore we can describe the system of particles by labeling them. Denote by $X_n(t)$ the position at time $t$ of particle $n$. Then we use the right-to-left convention, namely $X_{n+1}(t)<X_n(t)$ for any $n$ and $t$. Furthermore, we choose the labeling of the initial condition such that the particle with label $0$ is the right-most starting to the left of the origin, that is,
	\begin{equation}
		\cdots<X_1(0)<X_0(0)\leq 0<X_{-1}(0)<\cdots.
	\end{equation}
	We denote by $X(t)=(X_n(t))_{n\in\Z}$ the particles position process.
	
	\subsection{Backwards geodesic and index process}\label{sec22}
	Backwards geodesic and index process in the TASEP setting were introduced in \cite{Fer18}, see also \cite{FN19}. For TASEP they have many properties similar to geodesics in the last passage percolation models. Physically they track the space-time locations where the randomness is relevant for the position of the particles at time $t$.
	We extend its definition also to generic SEP, which includes ASEP and \newasep, and numerically investigate some (conjecturally universal) statistics.
	\begin{defin}[Backwards geodesic and index process]\label{dbwp}
		For any fixed $N\in\Z$ and $t\geq 0$, we call $(N(t\downarrow s))_{s\in [0,t]}$ as the \textbf{backwards index process} of particle $N$ starting from time $t$ and running backwards in time. It is defined as follows\footnote{In the notation, we do not explicitly write the dependence on $N$.}:
		\begin{enumerate}
			\item At time $t$, we set $N(t\downarrow t)=N$.
			\item The label changes at time $s\in [0,t)$ if the following occur:
			\begin{enumerate}
				\item there exists a suppressed right to left jump of particle $X_{N(t\downarrow s^+)}$, then we set $N(t\downarrow s)=N(t\downarrow s^+)-1$,
				\item there exists a suppressed left to right jump of particle $X_{N(t\downarrow s^+)}$, then we set $N(t\downarrow s)=N(t\downarrow s^+)+1$.
			\end{enumerate}
		\end{enumerate}
		The trajectory $(X_{N(t\downarrow s)})_{s\in[0,t]}$ is called the \textbf{backwards geodesic} of particle $N$ starting from time $t$. $N(t\downarrow 0)$ is called the \textbf{end-point of the backwards geodesic}.
	\end{defin}
	\begin{figure}[h]
		\centering
		\includegraphics{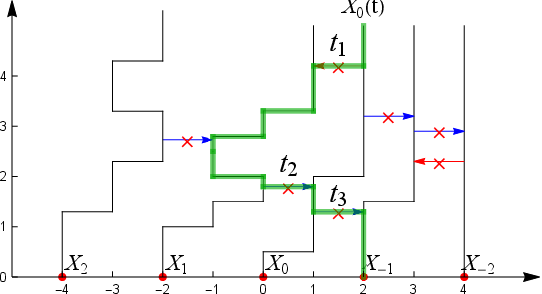}
		\caption{Illustration for backwards geodesic in ASEP setting for the index $N=0$. From left to right, the black solid lines are the space-time trajectory of particles $X_{2}$, $X_{1}$, $X_0,\ X_{-1}$ and $X_{-2}$ respectively. The arrow with a cross stands for a suppressed jump. The green line is the trajectory of backwards geodesic of particle $X_{0}$. $t_1,t_2$ and $t_3$ are the times when $N(t\downarrow s)$ changes its value, for instance, at time $t_2$, since particle $X_1$ has a right suppressed jump, we set $N(t\downarrow t_2)=N(t\downarrow t_2^+)-1$.
		}
	\end{figure}

	\paragraph{Construction of ASEP.} To state the properties of backwards index process, we recall the following construction of ASEP \cite{FG24}, on which Poisson clocks are assigned to particle labels. We have a family of independent Poisson processes $\{T_m^+,T_m^-\}_{m\in\Z}$ where $T_m^+$ has parameter $p$ and $T_m^-$ has parameter $1-p$. The points of the Poisson process $T_m^+$ (resp.\ $T_m^-$) are the time when a particle with label $m$ attempts to jump to its right (resp.\ left), and the jump is successful provided the arrival site is empty. By the independence of the Poisson processes, almost surely, there are no simultaneous jump trials. Moreover, for each fixed $t\geq 0$, almost surely there exist infinitely many positive and negative integers $m$ such that there exist no Poisson points of $T_m^+$ and $T_m^-$ in the time interval $[0,t]$ so that the construction is well-defined (in the same way Harris graphical construction for the basic coupling is well-defined~\cite{Har72,Har78,Li99}).	
	
	We couple two processes $X(t)$ and $Y(t)$ with different initial conditions $X(0)$ and $Y(0)$ such that particle $X_i$ and $Y_i$ use the same clocks $(T_i^+,T_i^-)$ for all $i\in\Z$. This is called the \textbf{clock coupling}~\cite{FG24}. For two configurations $X,\ Y$, we define a partial order $X\preceq Y$ if $X(i)\leq Y(i)$ for all $i\in\Z$. As for the basic coupling, the clock coupling also preserves the partial ordering.
	\begin{lem}[Lemma~3.9 of~\cite{FG24}]\label{attract}
		Under clock coupling, if $X(0)\preceq Y(0)$, then we have $X(s)\preceq Y(s)$ for all $s\geq 0$ almost surely.
	\end{lem}
	
	Next, we define a process, with step initial condition, coupled with $X$, in such a way that particle $n-M$ uses the same clocks as particle $n$ of the process $X$. To keep track of the coupling with $X$ and the shift of the clocks by $M$, we will call the new process $X^{\mathrm{step},X,M}$.
	\begin{defin}\label{comp}
		Let $X(t)$ be ASEP and $M\in\Z$ be fixed. We define the process $X^{{\rm step},X,M}(t)$ as follows:
		\begin{enumerate}
			\item Initial condition:
			\begin{equation}\label{ini}
				X^{\mathrm{step},X,M}_n(0)=\begin{cases}
					-n,\quad&\mathrm{if}\ n\geq 0,\\
					\infty,&\mathrm{otherwise}.
				\end{cases}
			\end{equation}
			\item Time evolution: for each $i\in\Z$, we let $X^{\mathrm{step},X,M}_{i-M}(t)$ and $X_i(t)$ use the same Poisson processes $(T_i^+,T_i^-)$.
		\end{enumerate}
	\end{defin}
For TASEP, it was shown in~\cite{Fer18,Sep98c} that
	\begin{equation}\label{tasep}
		X_N(t)=\min_{M\leq N}\left\{X_{N-M}^{\mathrm{step},X,M}(t)+X_M(0)\right\}=X_{N-N(t\downarrow0)}^{\mathrm{step},{X,{N(t\downarrow0)}}}(t)+X_{N(t\downarrow0)}(0)
	\end{equation}
	almost surely for all $N\in\Z$ and $t\geq 0$.
	\begin{rem}
		TASEP was known to satisfy the skew-time reversibility \cite{MQR17}, which states:
		\begin{equation}
			\begin{aligned}
				&\Pb\left(X_n(t)\geq s\right)=\Pb\left(\hat X_{n}(t)\leq X_n(0), \hat X_{n-1}(t)\leq X_{n-1}(0), \ldots, \hat X_{n-m}(t)\leq X_{n-m}(0)\right),
			\end{aligned}
		\end{equation}
		where $\hat X(t)$ is TASEP with left drift (i.e., $q=1$ and $p=0$) and initial condition
		\begin{equation}
			\hat X_r(0)=\begin{cases}
				s+(n-r),\quad & \text{if}\ r\leq n, \\
				-\infty, & \text{otherwise}.
			\end{cases}
		\end{equation}
		Note that the initial condition $\hat X(t)$ is a step initial condition. Hence, skew-time reversibility establishes a connection between general initial condition and step initial condition. Backwards geodesic provides the same property.
		
	\end{rem}
	
	 For ASEP with $p<1$, these equalities are not anymore true, but
	\begin{equation}
		X_{N-N(t\downarrow0)}^{\mathrm{step},{X,{N(t\downarrow0)}}}(t)+X_{N(t\downarrow0)}(0)\geq \min_{M\leq N}\left\{X_{N-M}^{\mathrm{step},X,M}(t)+X_M(0)\right\}\geq X_N(t).
	\end{equation}
	\begin{lem}\label{g0}
		Let $X(t)$ be ASEP with jump rate \eqref{aseprate}, then, almost surely,
		\begin{equation}\label{dnt1}
			D_N(t)=X_{N-N(t\downarrow0)}^{\mathrm{step},{X,{N(t\downarrow0)}}}(t)+X_{N(t\downarrow0)}(0)-X_N(t)\geq 0.
		\end{equation}
	\end{lem}
	\begin{proof}
		Let $Y(t)$ be a ASEP with initial condition
		\begin{equation}
			Y_n(0)=\begin{cases}
				X_n(0),\quad&\mathrm{if}\ n\geq N(t\downarrow 0),\\
				\infty,&\mathrm{otherwise},
			\end{cases}
		\end{equation}
		and  couple $Y(t)$ and $X(t)$ via clock coupling. By Lemma~\ref{attract}, we have almost surely $X(t)\preceq Y(t)$. Define process $Z(t)$ as
		\begin{equation}
			Z_m(t)=X_{m-N(t\downarrow0)}^{\mathrm{step},{X,{N(t\downarrow0)}}}(t)+X_{N(t\downarrow0)}(0).
		\end{equation}
		Then by Definition~\ref{comp}, we have
		\begin{equation}
			Z_m(0)\stackrel{\eqref{ini}}{=}-m\geq Y_m(0),\quad\forall m\geq N(t\downarrow0),
		\end{equation}
		and particle $Z_m$ and $Y_m$ use the same Poisson clock. Using Lemma~\ref{attract} once again, we  obtain
		\begin{equation}
			Z_m(t)\geq Y_m(t)\geq X_m(t),\quad\forall m\geq N(t\downarrow0),\ t\geq 0,
		\end{equation}
		from which \eqref{dnt1} follows.
	\end{proof}

	\subsection{Observables}\label{sec23}
	
	\subsubsection{Limit behaviour of the end-point of the backwards geodesic }
	
	In this paper, we consider flat non-random initial condition with density $1/2$, that is
	\begin{equation}\label{flat}
		X_n(0)=-2n,\quad\forall n\in\Z.
	\end{equation}
	The first observable we are interested in is the end-point of the backwards geodesic, that is, $X_{N(t\downarrow 0)}(0)$.
	
	In TASEP, backwards geodesic could be viewed as analog of geodesics in the last passage percolation models \cite{BBF21}. Therefore the law of $X_{N(t\downarrow 0)}(0)$ should converges in the scaling limit to the end-point of the geodesics for last passage percolation for the point-to-line problem. This is given by
	\begin{equation}\label{uhat}
		\hat u=\argmax_{u\in\R}\{{\cal A}_2(u)-u^2\},
	\end{equation}
 where ${\cal A}_2$ is Airy$_2$ process \cite{PS02}.

	In \cite{Jo03b} it is proven that\footnote{To be precise, it is proven under the assumption of uniqueness of $\hat u$, which was subsequently shown in \cite{CH11}.} $\lim_{N\to\infty}\mathcal T_N\stackrel{d}{=}\hat u$, where $\mathcal T_N$ is the endpoint of geodesic in a point-to-point last passage percolation (LPP) setting. We show a similar result for the end-point of the backwards geodesic in TASEP.
	\begin{thm}\label{convergence}
		Consider TASEP with flat initial condition \eqref{flat}. Then, for any $N\in\Z$, we have $\lim_{t\to\infty}B^{\rm TASEP}_t\stackrel{d}{=}\hat u$, where
		\begin{equation}\label{goal}
			B^{\rm TASEP}_t=\frac{X_{N(t\downarrow0)}(0)-X_N(0)-t/2}{2^{1/3}t^{2/3}}.
		\end{equation}
	\end{thm}
	Based on KPZ scaling theory \cite{Spo14}, which will be discussed in more detail in Section~\ref{kpz}, we get the following predictions for ASEP and \newasep.
	\begin{conj}\label{asepconj}
		Consider ASEP with flat initial condition \eqref{flat}. Then for any $N\in\Z$ we have $\lim_{t\to\infty}	B^{\rm ASEP}_t\stackrel{d}{=}\hat u$, where
		\begin{equation}
			B^{\rm ASEP}_t=\frac{X_{N(t\downarrow0)}(0)-X_N(0)-(p-q)t/2}{2^{1/3}((p-q)t)^{2/3}}.
		\end{equation}
	\end{conj}
	To state the conjecture for \newasep, we need additional assumptions on its jump rate, under which the invariant measure is known \cite{KS92} and it allows us to compute analytically all model-dependent parameters in the scaling conjecture. Fix parameters $\beta,E\in\R$ and choose the jump rates in \eqref{jp1}, \eqref{jp2}  as
	\begin{equation}\label{asepnjumprate}
		\begin{array}{lll}
			\alpha_2=1, &\alpha_3=\frac12\left(1+e^{- \beta}\right),&\alpha_4=e^{-\beta},\\[1em]
			\gamma_2=e^{-\beta}e^{-E}, &\gamma_3=\frac12\left(1+e^{-\beta}\right)e^{-E},& \gamma_4=e^{-E}.
		\end{array}
	\end{equation}

	\begin{conj}\label{conj}
		Consider \newasep, with flat initial condition \eqref{flat} and rates given by \eqref{asepnjumprate}. Then, for any $N\in\Z$, we have $\lim_{t\to\infty}	B^{{\rm ASEPsc}}_t\stackrel{d}{=}\hat u$, where
		\begin{equation}\label{basepsc}
			B^{{\rm ASEPsc}}_t=\frac{X_{N(t\downarrow0)}(0)-X_N(0)-2J(\beta,E)t}{2^{1/3}t^{2/3}\Gamma(\beta,E)^{2/3}A(\beta )^{-1}},
		\end{equation}
		where
		\begin{equation}\label{e122}
			\begin{aligned}
				J(\beta,E)&=\frac{1-e^{-E}}{2 \left(e^{\beta /2}+e^{\beta }\right)}, \quad A(\beta)=\frac{e^{\beta /2}}{4},\\
				\Gamma(\beta,E)&=\frac{\left(3 e^{\beta /2}-1\right)
					\left(1-e^{-E}\right)}{e^{\beta /2}+e^{\beta }}.
			\end{aligned}
		\end{equation}
	\end{conj}

    In Section~\ref{asepbkp} we present the numerical simulations confirming Conjecture~\ref{asepconj} and in Section~\ref{asepscbkp} for Conjecture~\ref{conj}.
	
	\subsubsection{Particle's position in speed changed ASEP}
	The KPZ scaling theory explained in Section~\ref{kpz} leads also to the following conjecture for the limiting distribution of particle's position in \newasep.
	\begin{conj}\label{conjp}
		For \newasep\ with initial condition \eqref{flat}. Define the rescaled process
		\begin{equation}\label{xresc}
			X^{\rm resc}_N(t)=\frac{X_N^{{\newasep}}(t)-X^{\newasep}_N(0)-2J(\beta,E)t}{-2\Gamma(\beta,E)^{1/3}t^{1/3}}.
		\end{equation}
		Then, for any $N\in\Z$, we have 	
		\begin{equation}\label{conjposition}
			\lim_{t\to\infty}\Pb\left(	X^{\rm resc}_N(t)\leq s\right)=F_{\rm GOE}(2s),\quad\forall s\in\R,
		\end{equation}
		where $F_{{\rm GOE}}$ is Tracy-Widom GOE distribution~\cite{TW96}.
	\end{conj}
	
    In Section~\ref{asepscpp} we numerically show that this conjecture holds true.

	\subsubsection{Discrepancy for ASEP}
	Recall that $D_N(t)$ defined in \eqref{dnt1} is generally not zero in ASEP, but positive. Moreover, $D_N(t)$ is an upper bound for the discrepancy between $\min_{M\leq N}\{X_{N-M}^{\mathrm{step},X,M}(t)+X_M(0)\}$ and $X_N(t)$. From recent results on ASEP, it follows that
\begin{lem}\label{lemScaling}
We have
	\begin{equation}
\lim_{t\to\infty}\frac{\min_{M\leq N}\{X_{N-M}^{\mathrm{step},X,M}(t)+X_M(0)\}-X_N(t)}{t^{1/3}} = 0
\end{equation}
in probability.
\end{lem}
\begin{proof}
	It is enough to show $\lim_{t\to\infty}D^{{\rm resc}}_N(t)\stackrel{d}{=}0$, where
	\begin{equation}\label{eq2.25}
		D^{{\rm resc}}_N(t):=\frac{\min_{M\leq N}\{X_{N-M}^{\mathrm{step},X,M}(t)+X_M(0)\}}{t^{1/3}}-\frac{X_N(t)}{t^{1/3}}.
	\end{equation}
	Recall that $X(t)$ is ASEP with flat initial condition with asymmetric parameter $p$. For $t>0,$ we define $\hat t:=(2p-1)t.$ Without loss of generality, we set $N=\hat t/4$. We have, see~\cite{OQR15,QS20},
	\begin{equation}\label{goecvg}
		\lim_{t\to \infty}\Pb\left(\frac{X_N(t)}{-\hat t^{1/3}}\leq s\right)=F_{{\rm GOE}}(2s).
	\end{equation}
	By \eqref{dnt1}, we have
	\begin{equation}\label{secondline}
		\begin{aligned}
			\frac{X_N(t)}{-\hat t^{1/3}}&\leq \max_{M\leq N}\frac{X_{N-M}^{{\rm step},X,M}+X_M(0)}{-\hat t^{1/3}}\\
			&\leq 2^{-1/3}\max_{u\geq 2^{-4/3}\hat t^{1/3}}\frac{X_{\hat t/4+2^{-2/3}\hat t^{2/3}u}^{{\rm step},X,-2^{-2/3}\hat t^{2/3}u}(t)+2^{1/3}\hat t^{2/3}u}{-2^{-1/3}\hat t^{1/3}},
		\end{aligned}
	\end{equation}
	where in the second step we parameterize $M$ by $-2^{-2/3}\hat t^{2/3}u$.

	In \cite{QS20}, it is showed that in the $t\to\infty$ limit the r.h.s.~of \eqref{secondline} converges to $2^{-1/3}T$ with $T=\max_{u\in\R}\left\{A_2(u)-u^2\right\}$. It is known that $\Pb\left(2^{-1/3}T\leq s\right)=F_{{\rm GOE}}(2s)$ \cite{Jo03b}.
	
	This means that, properly rescaled, both terms in the r.h.s.~of \eqref{eq2.25} converge to the same distribution as $t\to\infty$, and since $D^{{\rm resc}}_N(t)\geq 0$ one stochastically dominated the other. This implies that $D^{{\rm resc}}_N(t)$ scaled by $t^{-1/3}$ converges to $0$, see for instance Lemma~4.1 of~\cite{BC09}.
\end{proof}

If $N(t\downarrow 0)$ is close to the minimization index, then we would also expect that $D_N(t)/t^{1/3}\to 0$ as $t\to\infty$. However, it remains unclear whether the unscaled discrepancy $D_N(t)$ diverge or converge to a non-degenerate random variable as $t\to\infty$. Our numerical simulation indicates that the latter is the case.
	\begin{conj}\label{ConjDiscrepancy}
		There exists a non-degenerate (discrete) distribution $G$ such that
		\begin{equation}
			\lim_{t\to\infty}\Pb\left(D_N(t)\leq s\right)=G(s),\quad\forall s\in\R.
		\end{equation}
	\end{conj}

In Section~\ref{asepdis} we provide numerical evidence of Conjecture~\ref{ConjDiscrepancy}.

	
	\section{Analytic results}
	In this section we will show Theorem~\ref{convergence} and calculate the explicit formula mentioned in several conjectures above.
	\subsection{Proof of Theorem~\ref{convergence}}\label{pl22}
	In the whole section, we will assume $X(t)$ is a TASEP with $X_n(0)=-2n$ for all $n\in\Z$. For a fixed $N$, recall that $N(t\downarrow 0)$ is the end-point of the backwards geodesic of particle $X_N$ starting from time $t$ (see Definition~\ref{dbwp}). Define
	\begin{equation}
		\hat u=\argmax_{u\in\R}\{{\cal A}_2(u)-u^2\},
	\end{equation}
	where ${\cal A}_2$ is the $\mathrm{Airy}_2$ process. Since $X_n(0)=-2n$ for all $n\in\Z$, the statement of Theorem~\ref{convergence} is equivalent to $\lim_{t\to\infty}u_t\stackrel{d}{=}\hat u$, where
	\begin{equation}\label{dut}
		u_t=\frac{N-N(t\downarrow0)-t/4}{2^{-2/3}t^{2/3}}.
	\end{equation}
	
	The strategy is as follows. Let $X^{{\rm step}}(t)$ be TASEP with step initial condition, i.e., $X_n(0)=-n$ for all $n\geq 0$ and define the rescaled process
	\begin{equation}\label{e33}
		H_t(x)=\frac{X^{\mathrm{step}}_{\lfloor t/4+2^{-2/3}t^{2/3}x\rfloor}(t)+2^{1/3}t^{2/3}x}{-2^{-1/3}t^{1/3}},\quad\forall x\geq -2^{-4/3}t^{1/3}
	\end{equation}
	and
	\begin{equation}\label{dhatut}
		\hat u_t=\inf\big\{x\geq -2^{-4/3}t^{1/3}\big| H_t(x)=\max_{y} H_t(y)\big\}.
	\end{equation}
	In Lemma~\ref{connect}, we will show that, for any fixed $t>0$,
	\begin{equation}\label{same}
		u_t\stackrel{d}{=}\hat u_t.
	\end{equation}
	It is then enough to show $\hat u_t\to \hat u$ in distribution as $t\to\infty$.	By Proposition~2.9 in \cite{BBF21}, $H_t(u)\to {\cal A}_2(u)-u^2$ weakly on $C([-L,L])$ for any $0<L<\infty$. Together with the tightness of the distribution of $\hat u_t$ (see Lemma~\ref{ltight} below) and the uniqueness of $\argmax_{u\in\R}\{{\cal A}_2(u)-u^2\}$ (see for instance \cite[Theorem 4.3]{CH11}, \cite{MQR13},\cite{P15}), we conclude that $\hat u_t$ converges in distribution to $\hat u$ (see Proposition~\ref{pcvg} below). Together with \eqref{same}, we then obtain Theorem~\ref{convergence}.
	
\subsubsection{Proof of the identity \eqref{same}}
The starting point is the following observation.
	
	\begin{lem}\label{lm21}
		Let $ X(t) $ be an arbitrary TASEP with initial condition $ X(0) $, and  $ X^{\mathrm{step}, X, M} $ be defined as in Definition~\ref{comp}. Then, for any $ t \geq 0 $ and $ N \in \mathbb{Z} $, we have
		\begin{equation}\label{e26}
			N(t \downarrow 0) = \max\{n \leq N | X^{\mathrm{step}, X, n}_{N-n}(t) + X_n(0) = X_N(t)\}
		\end{equation}
see Figure~\ref{figlemma21} for an illustration.
	\end{lem}
	\begin{figure}[h]
		\centering
		\begin{subfigure}
			\centering
			\includegraphics[width=0.49\textwidth]{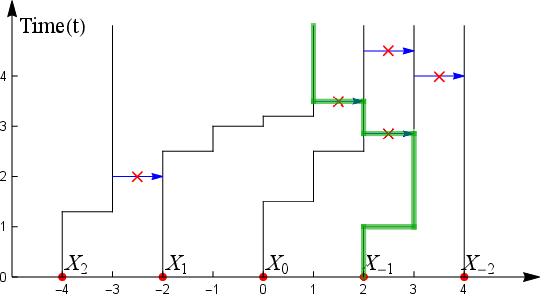}
		\end{subfigure}
		\begin{subfigure}
			\centering
			\includegraphics[width=0.49\textwidth]{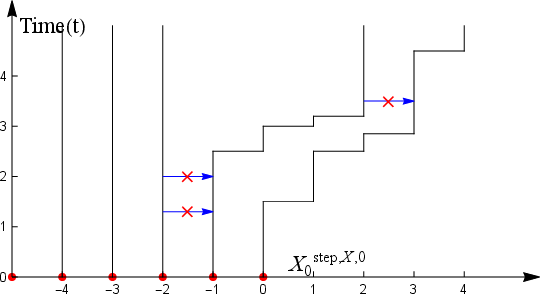}
		\end{subfigure}
		\caption{Illustration for Lemma~\ref{lm21}. On the left panel: from left to right, the black solid lines represent the trajectory of particles $X_2,\ X_1,\ X_0,\ X_{-1}$, and $X_{-2}$ respectively. The green line is the backwards trajectory of particle $X_1$ from time $t$. The end-point of this trajectory is $X_{-1}$. Note that $X_1(t=5)=1$. On the right panel: from right to left, the black solid lines represent the trajectories of particles $X_0^{{\rm step},X,0}$, $X_1^{{\rm step},X,0}$, and so on. In particular, note that $0 > -1 = N(t\downarrow 0)$ and $X_1^{{\rm step},X,0}(t=5) = 2 > 1 = X_1(t=5)$.}\label{figlemma21}
	\end{figure}
	\begin{proof}
		By \eqref{tasep}, $N(t\downarrow0)\in\{n\leq N| X^{{\rm step},X,n}_{N-n}(t)+X_n(0)=X_N(t)\}$. Let $M\in(N(t\downarrow 0),N]$ and we show that
		\begin{equation}
			X_{m-M}^{{\rm step},X,M}(t)+X_M(0)\geq X_m(t)+1,\quad\forall m\in[M,N].
		\end{equation}
		We define a new process $Z(t)$ with initial condition
		\begin{equation}
			Z_n(0)=\begin{cases}
				X_n(0),\quad&{\rm if}\ n\geq M,\\
				\infty,&{\rm otherwise}
			\end{cases}
		\end{equation}
		with $Z_n$ sharing the same Poisson clocks with $X_n$. By monotonicity property (see Lemma~\ref{attract}), we have $X_{m-M}^{{\rm step},X,M}(t)+X_M(0)\geq Z_m(t)$. By definition of backwards geodesic, for each $m\in[M,N]$, there exists
		\begin{equation}
			t_N>t_{N-1}>\cdots >t_M
		\end{equation}
		such that at time $t_m$ particle $X_m$ has a right suppressed jump by the presence of $X_{m-1}$. We claim that
		\begin{equation}
			Z_m(s)\geq X_m(s)+1,\quad\forall m\in[M,N],\ s\in(t_m,t].
		\end{equation}
		We show this via induction on $m$. For $m=M$, since $Z_{M-1}(s)=\infty$ for all $s\geq 0$, the attempted jump at $t_M$ will be a successful one for $Z_M$, hence, we have $Z_M(t_M^+)\geq X_M(t_M^+)+1$. This also implies $Z_M(s)\geq X_M(s)+1$ for all $s>t_M,$ since $Z_M$ and $X_M$ shares the same Poisson clock. Suppose now $Z_{N-1}(s)\geq X_N(s)+1$ for all $s>t_{N-1}$, then $Z_{N-1}(t_N)\geq X_{N-1}(t_N)+1$, thus, the attempted jump at $t_N$ will also be successful for $Z_N$ and hence $Z_N(t_N^+)\geq X_N(t_N^+)+1$. Using monotonicity property once more, we then obtain $Z_N(t)\geq X_N(t)+1$.
	\end{proof}
	
	Applying \eqref{tasep}, $X_n(0)=-2n$ for all $n\in\Z$ and Lemma~\ref{lm21}, we obtain
	\begin{equation}\label{e27}
		N(t\downarrow 0)=\max\left\{n\leq N\,\Big|\,  X^{{\rm step},X,n}_{N-n}(t)-2n=\min_{m\leq N}\left\{X^{{\rm step},X,m}_{N-m}(t)-2m\right\}\right\}.
	\end{equation}

	\begin{lem}\label{connect}
		For fixed $t\geq 0$, it holds $u_t\stackrel{d}{=}\hat u_t$.
	\end{lem}
	\begin{proof}
		Choose a fixed $t\geq 0$. By the spatial homogeneity, we only need to show the case for $N=\tfrac t4$. Define now $n(x)=2^{-2/3}t^{2/3}x$, $Z_m^n=X^{{\rm step},X,n}_m(t)$, and
		\begin{equation}
			W(t)=\min_{x\geq-2^{-4/3}t^{1/3}} \left\{Z^{{n(x)}}_{t/4+n(x)}(t)+2n(x) \right\}.
		\end{equation}
		Then by \eqref{e27}, we have
		\begin{equation}\label{e218}
			\begin{aligned}
				&N(t\downarrow 0)=\max\left\{n\leq t/4\,\Big|\,  Z^{n}_{t/4-n}(t)-2n=\min_{m\leq t/4} \left\{Z^{m}_{t/4-m}(t)-2m \right\} \right\}\\
				=&-2^{-2/3}t^{2/3}\min\left\{x\geq-2^{-4/3}t^{1/3} \,\Big|\,  Z^{{n(x)}}_{t/4+n(x)}(t)+2n(x)=W(t)\right\}\\
				\stackrel{d}{=}&-2^{-2/3}t^{2/3}\inf\left\{x\geq -2^{-4/3}t^{1/3}| H_t(x)=\max H_t(y)\right\}=-2^{-2/3}t^{2/3}\hat u_t,
			\end{aligned}
		\end{equation}
		where in the second last step we use the fact that $X^{{\rm step}}$ and $X^{{\rm step}, X,M}$ have the same distribution for all $M\leq N$.
	\end{proof}
	
	\subsubsection{Localization}
	Finally we need to show $\lim_{t\to\infty}\hat u_t\stackrel{d}{=}\hat u$. First, we establish the following tightness result for $\hat u_t$:
	\begin{lem}\label{ltight}
		For any $L>0$, there exists $T=T(L)>0$ such that
		\begin{equation}
			\Pb\left(|\hat u_t|> L\right)\leq Ce^{-cL}\quad\forall t\geq T(L)
		\end{equation}
for constants $C,c>0$ independent on $L$.
\begin{rem}
A similar result was previously obtained in the context of LPP; see (4.21) of \cite{FO17}. However, additional steps are required to translate this result into the setting of TASEP. Therefore, we provide a direct proof here, without relying on the connection to LPP.

\end{rem}

	\end{lem}
	The proof is postponed to the end of this section. Lemma~\ref{ltight} direcly implies
	\begin{prop}\label{pcvg}
		Let $\hat u_t$ be defined as \eqref{dhatut}, then
		\begin{equation}\label{c}
			\lim_{t\to\infty}\hat u_t\stackrel{d}{=}\hat u.
		\end{equation}
	\end{prop}
		\begin{proof}
		We define the functional $\sargmax$ \cite{Jo03b} on $C([-L, L])$ (without dependence on $L$ for notation) as:
		\begin{equation}
			\sargmax_{|x|\leq L} f(x)=\inf\Big\{u\geq -L\,\big|\, \max_{x\in [-L,u]} f(x)=\max_{x\in[-L,L]}f(x)\Big\}.
		\end{equation}
		Additionally,
		\begin{equation}
			\hat u_t^L=\sargmax_{|x|\leq L} H_t(x),\quad\textrm{and}\quad \hat u^L=\sargmax_{|x|\leq L} \left({\cal A}_2(u)-u^2\right).
		\end{equation}
		By the triangle inequality we obtain		
		\begin{equation}\label{triangle}
			\begin{aligned}
				&\left|\Pb\left(u_t \leq x\right)-\Pb(\hat{u} \leq x)\right| \\
				\leq& \left|\Pb\left(u_t^L \leq x\right)-\Pb\left(\hat{u}^L \leq x\right)\right| + \Pb\left(\hat{u}^L \neq \hat{u}\right) + \Pb\left(u_t \neq u_t^L\right).
			\end{aligned}
		\end{equation}
		Since the maximizer of $\mathcal{A}_2(u) - u^2$ is unique, it is a continuous point of the functional $\sargmax$. Thus, by weak convergence, the first term in the right-hand side of \eqref{triangle} converges to $0$ as $t\to\infty$. Moreover, \cite{Jo03b} shows that $\Pb\left(\hat{u} \neq \hat{u}^L\right) \leq \epsilon$ for large $L$. By Lemma~\ref{ltight}, the third probability is also arbitrarily small for sufficiently large $L$. Thus taking first $t\to\infty$ and then $L\to\infty$ we get that $|\eqref{triangle}|\leq 2\e$ for any $\e>0$, which concludes the proof.		
	\end{proof}
	We adapt the method for Proposition 1.4 in \cite{Jo03b} to prove Lemma~\ref{ltight}. Recall $H_t(u)$ defined in\footnote{We drop the 'step' in the notation in below.} \eqref{e33}. We define
	\begin{equation}
		S_t^\infty=\max_{u\geq 2^{-4/3}t^{1/3}}H_t(u),\quad S_t^L=\max_{u\in[-L,L]}H_t(u).
	\end{equation}

	\begin{lem}
		For any $L>0$, there exists $T=T(L)>0$ such that
		\begin{equation}
			\Pb\left(S_t^\infty\neq S_t^L\right)\leq Ce^{-cL}\quad\forall t\geq T(L).
		\end{equation}
	\end{lem}
	\begin{proof}
		Choose $\varepsilon>0$.	For arbitrary $A\in\R$, we have
		\begin{equation}
			\left\{H_t(0)>A\right\}\cap\Big\{\max_{u\not\in[-L,L]}H_t(u)< A\Big\}\subset\{S_t^\infty=S_t^L\}.
		\end{equation}
		Thus,
		\begin{equation}\label{e14}
			\Pb\left(S_t^\infty\neq S_t^L\right)\leq \Pb\left(H_t(0)\leq A\right)+\Pb\left(\max_{u\not\in[-L,L]}H_t(u)\geq A\right).
		\end{equation}
		Choosing $A=-2^{-3}L^2$, we have\footnote{Combining Proposition B.1 and (56) in \cite{BFP12}, one obtains the upper bound for LPP in point-to-point setting, i.e., $\Pb\left(L_{\ell}^{\mathrm{resc}} \leq s\right) \leq C e^{-c|s|^{3 / 2}},$ where $L_{\ell}^{\mathrm{resc}}$ is the rescaled last passage time from $(0,0)$ to $(\ell,\ell)$. Using the fact that the maximal last passage time from line $\mathcal{L}=\{(k,-k), k \in \mathbb{Z}\}$ to $(\ell,\ell)$ is larger than the one from $(0,0)$ to $(\ell,\ell)$, one also obtain the upper bound for LPP in point-to-line setting, see also Proposition B.1 (c) in \cite{FO22}.}
		\begin{equation}\label{e15}
			\Pb\left(H_t(0)\leq A\right)=\Pb\left(X_{t/4}\geq  2^{-3}L^2t^{1/3}\right)\leq  Ce^{-c L^3}.
		\end{equation}
		We can bound the second term on the right hand side of \eqref{e14} as
		\begin{equation}\label{e16}
			\Pb\left(\max_{u\not\in[-L,L]}H_t(u)\geq A\right)\leq \underbrace{\Pb\left(\max_{u\geq L}H_t(u)\geq A\right)}_{=:\operatorname{I}}+\underbrace{\Pb\left(\max_{u\in[ 2^{-4/3}t^{1/3},-L]}H_t(u)\geq A\right)}_{=:\operatorname{II}}.
		\end{equation}
		It remains to control the two terms on the right hand side.
		Using definition of $H_t(x)$, we have
		\begin{equation}\label{e17}
			\begin{aligned}
				&\operatorname{I}=\Pb\left(\min_{n\leq -2^{-2/3}t^{2/3}L}\{X_{t/4-n}(t)-2n\}\leq2^{-10/3} t^{1/3}L^2\right)\\
				=&\Pb\bigg(\min_{n\in[-\frac t4-L^2t^{1/3},-2^{-2/3}t^{2/3}L] }\{X_{t/4-n}(t)-2n\}\leq2^{-10/3} t^{1/3}L^2\bigg)\\
				\leq&\Pb\bigg(\min_{n\in[-\frac t4-L^2t^{1/3},-\frac t4] }\{X_{t/4-n}(t)-2n\}\leq 2^{-10/3}t^{1/3}L^2\bigg)\\
				+&\Pb\bigg(\min_{n\in[-\frac t4,-2^{-2/3}t^{2/3}L] }\{X_{t/4-n}(t)-2n\}\leq2^{-10/3} t^{1/3}L^2\bigg),
			\end{aligned}
		\end{equation}
		where the second equality follows from
		\begin{equation}
			\begin{aligned}
				&X_{t/4-n}(t)-2n\geq -n-\frac t4\geq L^2t^{1/3}
			\end{aligned}
		\end{equation}
		for all $n\leq -\frac t4-L^2t^{1/3}$ and $t\geq 0$, since $X(t)$ is TASEP with step initial condition. Setting $\gamma=\frac14$, $T=2t$ and $K=2^{-7/3}L$ in  (4.53) of \cite{FG24}, one obtains that the probability on the last line of \eqref{e17} is smaller than $Ce^{-cL}$. Note that
		\begin{equation}
			2^{-10/3}L^2t^{1/3}-2n\leq (1-2\sqrt{1/4+nt^{-1}})t-(n^2t^{-4/3}-\alpha^2)t^{1/3}
		\end{equation}
		for all $ n\in[Lt^{2/3},t/4+L^2t^{1/3}].$ The term on the second last line of \eqref{e17} is bounded by
		\begin{equation}\label{e111}
			\begin{aligned}
				&\sum_{n=t/4}^{t/4+L^2t^{1/3}}\Pb\left(X_{t/4+n}(t)\leq (1-2\sqrt{1/4+nt^{-1}})t-(n^2t^{-4/3}-L^2)t^{1/3}\right)\\
				\leq& CL^2t^{1/3}e^{-c(t^{2/3}/16-L^2)}\leq C e^{-c L^2}
			\end{aligned}
		\end{equation}
		for all $t\geq (8L)^3$. Thus $\operatorname{I}\leq Ce^{-cL}$ for some constants $C,c>0$. It remains to control $\operatorname{II}.$	Applying definition of $H_t(x)$, we have
		\begin{equation}
			\begin{aligned}
				&\operatorname{II}=\Pb\left(\min_{u\in[ 2^{-2/3}Lt^{2/3},\frac t4]}\left\{X_{t/4-n}(t)-2n\right\}\leq L^2t^{1/3}\right).
			\end{aligned}
		\end{equation}
		Let $X^{{\rm hf}}(t)$ be TASEP with half flat initial condition, i.e.,
		\begin{equation}
			X^{{\rm hf}}_n(0)=\begin{cases}
				-2n,\quad&\mathrm{if}\ n\geq 0,\\
				\infty,&\mathrm{otherwise}.
			\end{cases}
		\end{equation}
		By \eqref{tasep}, we have
		\begin{equation}
			\begin{aligned}
				&X^{{\rm hf}}_{t/4-2^{-2/3}Lt^{2/3}}(t)=\min_{M\in[0,t/4-2^{-2/3}Lt^{2/3}]}\left\{X^{{\rm step},X^{{\rm hf}},M}_{t/4-2^{-2/3}Lt^{2/3}-M}(t)-2M\right\}\\
				=&\min_{n\in[ 2^{-2/3}Lt^{2/3},\frac t4]}\left\{X^{{\rm step},X^{{\rm hf}},M}_{t/4-n}(t)-2n\right\}-2^{1/3}Lt^{2/3}.
			\end{aligned}
		\end{equation}
		In other words, there exists a coupling between $X(t)$ and $X^{{\rm hf}}(t)$ such that
		\begin{equation}
			\min_{n\in[ 2^{-2/3}Lt^{2/3},\frac t4]}\left\{X_{t/4-n}(t)-2n\right\}=X^{{\rm hf}}_{t/4-2^{-2/3}Lt^{2/3}}(t)-2^{1/3}Lt^{2/3}.
		\end{equation}
		Then we have\footnote{This comes from the proof of Theorem 2.6 in \cite{CFS16}. To see this, we set $\ell=t$, $s=0$ and $kM=2^{2/3}\alpha$, then (2.33) in \cite{CFS16} will become \eqref{e117} in our setting.}
		\begin{equation}\label{e117}
			\operatorname{II}=\Pb\left(X^{{\rm hf}}_{t/4-2^{-2/3}Lt^{2/3}}(t)-2^{1/3}Lt^{2/3}\leq L^2t^{1/3}\right)\leq Ce^{-cL}.
		\end{equation}
		This finishes the proof.
		
	\end{proof}
	\begin{proof}[Proof of Lemma~\ref{ltight}]
		Note that $\{\hat u_t>L\}\subset\{S_t^L\neq S_t^\infty\}$, hence we have
		\begin{equation}
			\Pb\left(\hat u_t>L\right)\leq Ce^{-cL}
		\end{equation}
		for $t$ large. Also, we have $\{\hat u_t<-L\}\subset \{\max_{u\in[ 2^{-4/3}t^{1/3},-L]}H_t(u)\geq A\}\cup\{H_t(0)\leq A\}$ for arbitrary $A$. Choosing $A$ as the one in proof of the previous lemma, we have
		\begin{equation}
			\Pb\left(\hat u_t<-L\right)\leq Ce^{-cL}
		\end{equation}
		by \eqref{e15} and \eqref{e117}.
	\end{proof}

	\subsection{KPZ scaling theory}\label{kpz}
	The KPZ scaling theory \cite{Spo14}, see also further details in Section~6 of~\cite{PS01}, provides a universal scaling formula for models in the KPZ class. After presenting the formulas conjectures by the KPZ sclaing theory, we compute the universal constants using results from TASEP. The model-dependent quantities are easy to compute for ASEP, but more involved for \newasep.
	
	Let $\eta_t$ be a general exclusion process with generator \eqref{jumpra}.
	\begin{asmp}[\cite{Spo14}]\label{spo}
		The spatially ergodic and time stationary measures of the process $\eta_t$ are precisely labeled by the average density
		\begin{equation}
			\rho=\lim _{a \rightarrow \infty} \frac{1}{2 a+1} \sum_{|j| \leq a} \eta(j)
		\end{equation}
		with $|\rho| \leq 1$.
	\end{asmp}
	Denote $\mu_\rho$ as the stationary measure satisfying Assumption~\ref{spo}. Define the stationary current as
	\begin{equation}\label{current1}
		\begin{aligned}
			J(\rho)=&\mu_\rho\left(c_\eta(0,1)\eta(0)(1-\eta(1))\right)-\mu_\rho\left(c_\eta(1,0)\eta(1)(1-\eta(0))\right)
		\end{aligned}
	\end{equation}
	and integrated stationary covariance as
	\begin{equation}\label{covariance}
		A(\rho)=\sum_{j\in\Z}(\mu_\rho\left(\eta(0)\eta(j)\right)-\mu_\rho\left(\eta(0)\right)^2)=\sum_{j\in\Z}(\mu_\rho\left(\eta(0)\eta(j)\right)-\rho^2).
	\end{equation}
	Set $\Gamma(\rho)=-A(\rho)^2 J''(\rho)$.

	\subsubsection{Scaling for particle's position}
	As explained in~\cite{PS01}, the fluctuations of the integrated current (height function) in KPZ models should be scaled by $\Gamma(\rho)^{1/3}t^{1/3}$ and the spatial scaling around the macroscopic behaviour should be scaled by $\Gamma(\rho)^{2/3} A(\rho)^{-1} t^{2/3}$. The height function is defined by setting $h(0,0)=0$, $h(0,t)=2 J(t)$, with $J(t)$ then integrated particle current at $0$ in the time interval $[0,t]$ and $h(x+1,t)-h(x,t)=1-2\eta_t(x)$.

For initial condition $x_n(0)=-\lfloor n/\rho\rfloor$, by universality we expect that the rescaled height function
	\begin{equation}
		\frac{h(\xi t+c_1\Gamma(\rho)^{2/3}A(\rho)^{-1}t^{2/3}u,t)-(1-2\rho)(\xi t+c_1\Gamma(\rho)^{2/3}A(\rho)^{-1}t^{2/3}u)-2J(\rho)t}{-c_2\Gamma(\rho)^{1/3}t^{1/3}}
	\end{equation}
	converges weakly to ${\cal A}_1(u)$ for some model independent constants $c_1,c_2.$

This result can be restated in terms of particle's positions by using the identity
	\begin{equation}\label{e341}
		\Pb\left(h(m-n,t)\geq m+n\right)=\Pb\left(X_n(t)\geq m-n\right).
	\end{equation}
A simple computation leads to the following scaling formula for particle's position.
	\begin{conj}\label{kpzconjecture}
		Suppose $\eta_t$ is Markov process with generator \eqref{jumpra}, initial condition $X_n(0)=-\lfloor n/\rho\rfloor$ with $\rho\in(0,1)$, and invariant measure $\mu_\rho$ satisfying Assumption~\ref{spo}.
		Then there exist universal constants $c_1$ and $c_2$ such that
		\begin{equation}\label{a18}
			\lim_{t\to\infty}\frac{X_{u\theta(\rho) t^{2/3}}(t)+u\theta(\rho) t^{2/3}/\rho-J(\rho)t/\rho}{-(c_2\Gamma(\rho)t)^{1/3}/\rho}={\cal A}_1(u),
		\end{equation}
		where $\theta(\rho)=c_1\Gamma(\rho)^{2/3} A(\rho)^{-1}\rho$.
	\end{conj}
    The constants are $c_1=2$ and $c_2=1$. These are obtained by considering the case of TASEP with flat initial condition \eqref{flat} it is proven that~\cite{BFPS06}
\begin{equation}
\lim_{t\to\infty}\frac{X_{[u t^{2/3}]}(t)+2 u t^{2/3}-t/2}{-t^{1/3}}={\cal A}_1(u).
\end{equation}
Since for TASEP $J(\rho)=A(\rho)=\rho(1-\rho)$, for $\rho=1/2$ we obtain $J(\rho)t/\rho=t/2$, $J''(\rho)=-2$, $\Gamma(\rho)=1/8$. Therefore we have $c_1=2$ and $c_2=1$. The scaling also fits with the results for flat TASEP with generic density $\rho$ as computed in Appendix~A of~\cite{FN23}.

The stationary measures of ASEP satisfy Assumption~\ref{spo} and below we show that the same holds true for \newasep.

	\subsubsection{Scaling for end-point of backwards geodesic}	
	In this section, we deduce the conjecture on the limit fluctuation behaviour of end-point of backwards geodesic. There are essentially two terms appearing in the formula: the law of large number term and the fluctuation order term. For the law of large term, we have
	\begin{equation}
		X_{N(t\downarrow 0)}(0)-X_N(0)=\underbrace{X_{N}(t)-X_N(0)}_{=:\mathrm{I}}-\underbrace{(X_N(t)-X_{N(t\downarrow 0)}(0))}_{=:\mathrm{II}}.
	\end{equation}
	By Definition, $J(\rho)$ is the expected rate of numbers of particles jumping across edge $\{0,1\}$, together with the density $\rho$, the law of large number term of ${\rm I}$ should be given as $J(\rho)t/\rho$. As for term ${\rm II},$ let $\rho(x,t)$ be the macroscopic particle's density, then it should governed by \cite[Equation (6.9)]{KS92}
	\begin{equation}
		\frac{\partial \rho}{\partial t}=-\frac{\partial J(\rho)}{\partial x}=-\frac{\partial J(\rho)}{\partial \rho}\cdot\frac{\partial\rho}{\partial x}
	\end{equation}
	so that $\rho$ is constant along the characteristic line $w(t)$ with $w(0)=\rho_0$ and $\frac{\partial w}{\partial t}=\frac{\partial J}{\partial \rho}$ \cite{Fer16}. On the other hand, the backwards geodesic mimics the characteristic line, thus, the law of large number term should be expressed as $J(\rho)t/\rho-J'(\rho)t$, see Figure~\ref{characteristicLine}.
	\begin{figure}[h]
		\centering
		\includegraphics{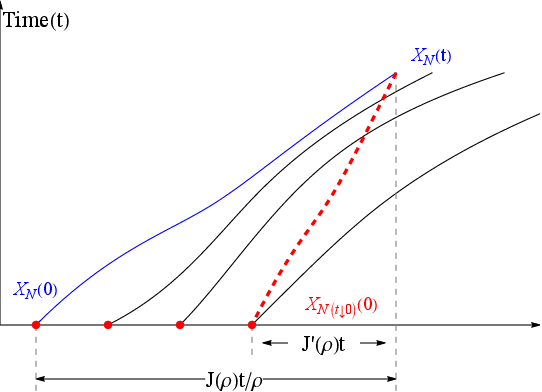}
		\caption{The solid lines represent particle trajectories, with the blue line being the trajectory from which we construct the backward geodesic (the dashed red line). We have $X_{N(t\downarrow 0)}(0) - X_N(0) = X_N(t) - X_N(0) - (X_N(t) - X_{N(t\downarrow 0)}(0))$. Macroscopically, the first difference is about $\frac{J(\rho) t}{\rho}$, while the second difference is given by $J'(\rho) t$.	}
\label{characteristicLine}
	\end{figure}
For the fluctuation order term, by Conjecture~\ref{kpzconjecture}, the correct order should be
	\begin{equation}\label{eq3.45}
		C\Gamma(\rho)^{2/3}A(\rho)^{-1}t^{2/3}.
	\end{equation}
for some $C\in\R$. Comparing to the known result of TASEP, i.e., Theorem~\ref{convergence}, we then obtain $C=2^{1/3}.$
	\begin{conj}\label{12}
		Consider any exclusion process with generator \eqref{jumpra} and initial condition $X_n(0)=-\lfloor n/\rho\rfloor$ for some $\rho\in(0,1)$ such that its translation invariant stationary measure $\mu_\rho$ satisfies Assumption~\ref{spo}. Then\footnote{In both ASEP and \newasep, we have $J'(\rho)=0$ for $\rho=1/2$.} there exists universal constant $c_2$ such that
		\begin{equation}\label{fluc}
			\lim_{t\to\infty}\frac{X_{N(t\downarrow 0)}(0)-X_N(0) -(J(\rho)/\rho-J'(\rho))t}{2^{1/3}\Gamma(\rho)^{2/3}A(\rho)^{-1}t^{2/3}}
			\stackrel{d}{=}\hat u,
		\end{equation}
		where $\hat u=\argmax_{u\in\R}\left\{{\cal A}_2(u)-u^2\right\}$.
	\end{conj}

	\subsubsection*{Scaling coefficients for the conjecture on ASEP}
	For ASEP, we can calculate
	\begin{equation}
		J(\rho)=(p-q)\rho(1-\rho),\quad A(\rho)=\rho(1-\rho),\quad\Gamma(\rho)=(p-q)/8.
	\end{equation}
	Plugging $\rho=\tfrac12$, we then obtain Conjecture~\ref{asepconj}.
	
	\subsubsection*{Scaling coefficients for the conjecture on speed changed ASEP}	
	For any $\rho\in[0,1]$, there exists an explicit stationary measure $\mu_\rho$ for \newasep\ as long as its jump rate satisfies
	\begin{equation}\label{db}
		\begin{array}{lll}
			\alpha_2=e^\beta \alpha_4, &\alpha_2+\alpha_4=2 \alpha_3,&\\[1em]
			\gamma_2=\alpha_2 e^{-\beta} e^{-E}, &\gamma_3=\alpha_3 e^{-E},& \gamma_4=\alpha_4 e^\beta e^{-E},
		\end{array}
	\end{equation}
	for some $\beta,E\in\R$. Moreover, the stationary measure $\mu_\rho$ on $\{0,1\}^\Z$ of $\eta_t$ has the form~\cite[(6.19)]{KS92}
	\begin{equation}\label{stationary} \mu_{\rho}(\eta)=\frac{1}{Z}e^{\beta\sum_{i}\eta(i)\eta(i+1)+h\sum_{i}\eta(i)},
	\end{equation}
	where $Z$ is the normalization constant and $h$ fixes the density. Defining
	\begin{equation}
		\mathfrak{h}(x,y)=e^{\beta xy+\frac{h}{2}(x+y)}
	\end{equation}
	we then have
	\begin{equation}\label{exp}
		\mu_{\rho}(\eta)=\frac1Z\prod_{i}\mathfrak h(\eta(i),\eta(i+1)).
	\end{equation}
	
	For initial condition with density $\rho=\tfrac12$, $\mu_\rho$ satisfies Assumption~\ref{spo}. Recall that we set
	\begin{equation}\label{newalpha}
		\begin{array}{lll}
			\alpha_2=1, &\alpha_3=\frac12\left(1+e^{- \beta}\right),&\alpha_4=e^{-\beta},\\[1em]
			\gamma_2=e^{-\beta}e^{-E}, &\gamma_3=\frac12\left(1+e^{-\beta}\right)e^{-E},& \gamma_4=e^{-E}.
		\end{array}
	\end{equation}
	One can check that those parameter indeed satisfies condition \eqref{db}. Hence, it remains to calculate $J(\rho)$ and $A(\rho)$. The choice of this model is due to the fact that, as mentioned in \cite{KS92}, it is one of the few growth model for which these model-dependent parameters can be computed analytically.
	
	\begin{lem}\label{asepn}
		Let $\beta,E\in\R$ and $\rho\in [0,1]$ and jump rate given in \eqref{newalpha}. We have
		\begin{equation}
			J(\rho,\beta,E)=J_{+}(\rho,\beta)(1-e^{-E}),
		\end{equation}
		where the positive current $J_+$ is given by
		\begin{equation}\label{positivecurrent}
			J_+(\rho,\beta)=\frac{2 \rho(1-\rho)\left[e^{{\beta}}+\sqrt{(1-2 \rho)^2+4 \rho(1-\rho)  e^{{\beta}}}\right]}{e^{{\beta}} \left[1+\sqrt{(1-2 \rho)^2+4 \rho(1-\rho) e^{{\beta}}}\right]^2}.
		\end{equation}
		Its integrated covariance is given by
		\begin{equation}
			A(\rho,\beta,E)=\rho(1-\rho)\sqrt{(1-2\rho)^2+4\rho(1-\rho)e^{\beta}}.
		\end{equation}
	\end{lem}
	Plugging this with $\rho=\tfrac12$ into Lemma~\ref{asepn} leads to \eqref{e122}. Applying Conjecture~\ref{kpzconjecture} (resp.\ Conjecture~\ref{12}), we then obtain Conjecture~\ref{conjp} (resp.\ Conjecture~\ref{conj}). Hence, it remains to prove Lemma~\ref{asepn}. First note that the measure $\mu_\rho$ in \eqref{exp} satisfies spatial Markov property.
	\begin{lem}
		For any $n\in\Z_{\geq 1}$ and $\eta\in\{0,1\}^\Z$, it holds
		\begin{equation}\label{spacemarkov}
			\mu_{\rho}(\eta(n)|\eta(0),\ldots,\eta(n-1))=\mu_{\rho}(\eta(n)|\eta(n-1)).
		\end{equation}
	\end{lem}
	\begin{proof}
		Let $m,n\in\Z$ with $m\leq n$ and $\eta(m),\ldots,\eta(n)\in\{0,1\}$. We also define \begin{equation}
			\begin{aligned}
				&\Z_{<n}=\{x\in\Z| x<n\},\quad \Z_{>n}=\{x\in\Z| x>n\},\quad\forall n\in\Z.
			\end{aligned}
		\end{equation}
		By \eqref{exp}, we have
		\begin{equation}
			\begin{aligned}
				&\mu_{\rho}(\eta(m),\ldots,\eta(n))\\
				=&	\frac1Z\sum_{\xi_1\in\{0,1\}^{\Z_{<m}}}\sum_{\xi_2\in\{0,1\}^{\Z_{>n}}}\left\{
				\bigg(\prod_{i<m-1}\mathfrak h(\xi_1(i),\xi_1(i+1))\!\bigg)\mathfrak h(\xi_1(m-1),\eta(m))\right.\\
				\times&\left.\bigg(\prod_{i=m}^{n-1}\mathfrak h(\eta(i),\eta(i+1))\!\bigg)\mathfrak h(\eta(n),\xi_2(n+1))
				\prod_{i>n}\mathfrak h(\xi_2(i),\xi_2(i+1))\right\}
			\end{aligned}
		\end{equation}
which can be further written as
		\begin{equation}
			\begin{aligned}
&\underbrace{\frac{1}{\sqrt{Z}}\sum_{\xi_1\in\{0,1\}^{\Z_{<m}}}\bigg(\prod_{i<m-1}\mathfrak h(\xi_1(i),\xi_1(i+1))\!\bigg)\mathfrak h(\xi_1(m-1),\eta(m))}_{=:L(m,\eta(m))}\\
&\times \bigg(\prod_{i=m}^{n-1}\mathfrak h(\eta(i),\eta(i+1))\!\bigg)\underbrace{\frac{1}{\sqrt{Z}}\sum_{\xi_2\in\{0,1\}^{\Z_{>n}}}\!\!\!\mathfrak h(\eta(n),\xi_2(n+1))\bigg(
					\prod_{i>n}\mathfrak h(\xi_2(i),\xi_2(i+1))\!\bigg)}_{=:R(\eta(n),n)}.
			\end{aligned}
		\end{equation}
		Thus we have
		\begin{equation}
			\begin{aligned}
				&\mu_{\rho}(\eta(n)| \eta(0),\ldots,\eta(n-1))=\frac{\mu_{\rho}(\eta(0),\ldots,\eta(n))}{\mu_{\rho}(\eta(0),\ldots,\eta(n-1))}\\
				=&h(\eta(n-1),\eta(n))\frac{R(\eta(n),n)}{R(\eta(n-1),n-1)}=\frac{\mu_{\rho}(\eta(n-1),\eta(n))}{\mu_{\rho}(\eta(n-1))}=\mu_{\rho}(\eta(n)|\eta(n-1)),
			\end{aligned}
		\end{equation}
		which completes the proof.
	\end{proof}
	This result implied in particular that
	\begin{equation}\label{spm} \mu_{\rho}(\eta(m),\ldots,\eta(n))=\mu_{\rho}(\eta(m))\prod_{i=m}^{n-1}\mu_{\rho}(\eta(i+1)|\eta(i)),\quad\forall m,n\in\Z
	\end{equation}
	allowing us to compute $J(\rho)$ and $A(\rho)$ as follows.
	
	\begin{proof}[Proof of Lemma~\ref{asepn}]
		\textbf{The stationary current.} By its definition, see \eqref{current1}, we have
		\begin{equation}
			J(\rho,\beta,E)=\underbrace{\sum_{\substack{\eta\in\Omega\\\eta(0)=1,\eta(1)=0}} c_{\eta}(0,1)\mu_\rho(\eta)}_{=:J_+(\rho,\beta)}-\underbrace{\sum_{\substack{\eta\in\Omega\\\eta(0)=0,\eta(1)=1}} c_{\eta}(1,0)\mu_\rho(\eta)}_{=:J_{-}(\rho,\beta)}.
		\end{equation}
		By the definition of jump rate \eqref{jp1}, we have
		\begin{equation}\label{positive}
			\begin{aligned}
				J_+(\rho,\beta)& =\alpha_2\mu_\rho(\eta(-1)=0,\eta(0)=1,\eta(1)=0,\eta(2)=1)\\
				&+\alpha_3\mu_\rho(\eta(-1)=1,\eta(0)=1,\eta(1)=0,\eta(2)=1)\\
				&+\alpha_3\mu_\rho(\eta(-1)=0,\eta(0)=1,\eta(1)=0,\eta(2)=0)\\
				&+\alpha_4\mu_\rho(\eta(-1)=1,\eta(0)=1,\eta(1)=0,\eta(2)=0).
			\end{aligned}
		\end{equation}
		Set
		\begin{equation}\label{alphahat}
			\hat\alpha=\mu_\rho(\eta(1)=1| \eta(0)=0),\quad \hat\beta =\mu_\rho(\eta(1)=0| \eta(0)=1).
		\end{equation}
		Applying~\eqref{newalpha}, translation invariance and spatial Markov property \eqref{spm} to~\eqref{positive}, we obtain
		\begin{equation}\label{jplus}
			\begin{aligned} J_+(\rho,\beta)=&(1-\rho)\hat\alpha\hat\beta\hat\alpha+\frac{1+e^{-\beta}}{2}\rho(1-\hat\beta)\hat\beta\hat\alpha\\ &+\frac{1+e^{-\beta}}{2}(1-\rho)\hat\alpha\hat\beta(1-\hat\alpha)+\rho(1-\hat\beta)\hat\beta(1-\hat\alpha)e^{-\beta}.
			\end{aligned}
		\end{equation}
		Similarly, we  get $J_{-}(\rho,\beta)=e^{-E}J_{+}(\rho,\beta)$, which then implies
		\begin{equation}\label{j}
			J(\rho,\beta,E)=(1-e^{-E})J_+(\rho,\beta).
		\end{equation}
		To show \eqref{positivecurrent}, we need to express $\hat\alpha,\hat\beta$ in terms of $\rho$ and $\beta$. Note that
		\begin{equation}
			\begin{aligned}
				1-\rho&= \mu_{\rho}(\eta(k+1)=0)\\
				&=  \mu_{\rho}(\eta(k+1)=0|\eta(k)=1)\mu_{\rho}(\eta(k)=1)\\
				&\quad+\mu_{\rho}(\eta(k+1)=0|\eta(k)=0)\mu_{\rho}(\eta(k)=0)\\
				&=\hat\beta\rho+(1-\hat\alpha)(1-\rho),
			\end{aligned}
		\end{equation}
		which then implies
		\begin{equation}\label{r}
			\rho=\frac{\hat\alpha}{\hat\alpha+\hat\beta}.
		\end{equation}
		On the other hand, by spatial Markov property~\eqref{spm}, we have
		\begin{equation}
			\begin{aligned} &\mu_{\rho}(\eta(0),\eta(1),\ldots)=\mu_{\rho}(\eta(0))\prod_{i=0}^\infty\mu_{\rho}(\eta(i+1)|\eta(i))\\
				=&(1-\rho)^{1-\eta(0)}\rho^{\eta(0)}\prod_{i=0}^\infty (1-\hat\alpha)^{(1-\eta(i))(1-\eta(i+1))}\hat\alpha^{(1-\eta(i))\eta(i+1)}\hat\beta^{\eta(i)(1-\eta(i+1))}(1-\hat\beta)^{\eta(i)\eta(i+1)}\\
				=&{\rm const}\times\prod_{i=0}^\infty \bigg[\frac{\hat\alpha\hat\beta}{(1-\hat\beta)(1-\hat\alpha)}\bigg]^{\eta(i)\eta(i+1)}\bigg[\frac{\hat\alpha\hat\beta}{(1-\hat\alpha)^2}\bigg]^{\eta(i)}.
			\end{aligned}
		\end{equation}
		Comparing this with \eqref{stationary} we obtain
		\begin{equation}\label{minusbeta}
			e^{-\beta}=\frac{\hat\alpha\hat\beta}{(1-\hat\beta)(1-\hat\alpha)}.
		\end{equation}
		Combining (\ref{r}) and (\ref{minusbeta}), we finally get
		\begin{equation}\label{final}
				\hat\alpha=\frac{2\rho}{1+\sqrt{(1-2\rho)^2+4\rho(1-\rho)e^{\beta}}},\quad
				\hat\beta=\frac{1-\sqrt{(1-2\rho)^2+4\rho(1-\rho)e^\beta}}{2\rho(1-\gamma)}.
		\end{equation}
		Plugging \eqref{final} back into (\ref{positive}) and (\ref{j}), we get the claimed expression for the stationary current.
		
		\textbf{The integrated stationary covariance.} Applying translation invariance and the definition of the integrated covariance, see~\eqref{covariance}, we have
		\begin{equation}\label{back}
			\begin{aligned}
				A(\rho,\beta,E)=&\sum_{j\in\mathbb Z}(\mu_\rho(\eta(0)\eta(j))-\rho^2)=\sum_{j\in\mathbb Z}(\mu_\rho(\eta(0)=\eta(j)=1)-\rho^2)\\
				=&\rho-\rho^2+2\sum_{j\geq 1}^\infty (\mu_\rho(\eta(0)=\eta(j)=1)-\rho^2)\\
				=&\rho(1-\rho)+2\sum_{j\geq 1}^\infty (\mu_\rho(\eta(j)=1|\eta(0)=1)\rho-\rho^2),
			\end{aligned}
		\end{equation}
		Choose a fixed $j>0$, by~\eqref{spacemarkov}, we can consider $(\eta(n))_{n\in\Z_{>0}}$ as a spatial Markov chain with transition matrix
		\begin{equation}
			T=\begin{pmatrix}
				1-\hat\alpha &\hat\alpha\\
				\hat\beta &1-\hat\beta
			\end{pmatrix},
		\end{equation}
		where $\hat\alpha,\hat\beta$ are the same as~\eqref{alphahat}. Then we have
		\begin{equation}
			\mu(\eta(j)=1|\eta(0)=1)=[T^j]_{1,1},
		\end{equation}
		where  $[T^j]_{1,1}$ is the component of matrix $T^{j}$ on first column and row. Applying eigenbasis decomposition, we have $T=U\hat T U^{-1}$, where
		\begin{equation}
			U=\begin{pmatrix}
				1 &-\hat\alpha\\
				1 &\hat\beta
			\end{pmatrix},\quad \hat T=\begin{pmatrix}
				1 &0\\
				0 &1-\hat\alpha-\hat\beta
			\end{pmatrix},
		\end{equation}
		which implies
		\begin{equation}
			\mu(\eta(j)=1|\eta(0)=1)=\rho+(1-\rho)(1-\hat\alpha-\hat\beta)^j,\quad\forall j\in Z_{\geq 1}.
		\end{equation}
		Plugging this into~\eqref{back}, we have
		\begin{equation}
			\begin{aligned}
				A(\rho,\beta,E)=&\rho(1-\rho)+2\sum_{j=1}^\infty(1-\rho)(1-\hat\alpha-\hat\beta)^{j}\\				=&\rho(1-\rho)\bigg(1+\frac{2(1-\hat\alpha-\hat\beta)}{\hat\alpha+\hat\beta}\bigg).
			\end{aligned}
		\end{equation}
		Replacing $\hat\alpha,\hat\beta$ we obtained the claimed formula.
	\end{proof}
	
	\section{Simulation results}\label{sectSimulation}
	In this section, we present the numerical results. The raw data can be found in the BonnData repository~\cite{FK2/PESMFT_2024}. Recall that for a random variable $Y$, the skewness and kurtosis are defined as
	\begin{equation}
		\mathrm{Skew}(Y) = \frac{\E\left((Y-\mu)^3\right)}{\E\left((Y-\mu)^2\right)^{3/2}}, \quad \mathrm{Kurt}(Y) = \frac{\E\left((Y-\mu)^4\right)}{\E\left((Y-\mu)^2\right)^{2}},
	\end{equation}
	where $\mu = \E(Y)$. Moreover, we denote
	\begin{equation}
		\overline{\E}(Y) = \frac{1}{n}\sum_{i=1}^n y_i,
	\end{equation}
	as its empirical expectation (we also use similar notation for other statistics: variance, skewness, and kurtosis), where $y_1,y_2,\ldots,y_n$ are independent realizations of the random variable $Y$.
	
	\subsection{Results for ASEP}
	\subsubsection{End-point of backwards geodesic}\label{asepbkp}
	
	We have simulated ASEP with $p = 3/4$, which is far away from the TASEP case ($p = 1$) and the SSEP (symmetric simple exclusion process) case ($p = 1/2$). In Figure~\ref{asepright}, we compare the empirical density function of $B_t^{{\rm ASEP}}$ at times $t=1000$ and $t=5000$ with the density function of $\hat u$. For increasing time, the empirical density function of $B_t^{{\rm ASEP}}$ converges to the density function of $\hat u$.
	\begin{figure}[h]
		\centering
		\includegraphics[height=5cm]{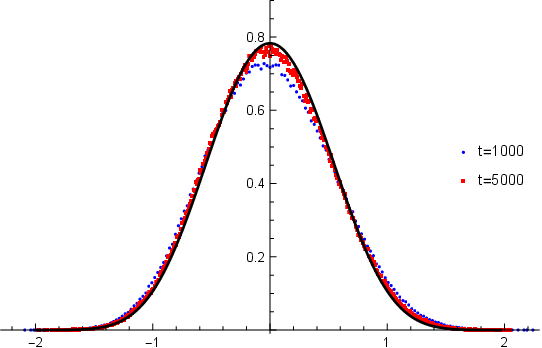}
		\caption{Comparison between the density function of $B_t^{\rm ASEP}$ and $\hat u$. The black line represents the probability density function of $\hat u$, while the red (resp. blue) points correspond to the probability density of $B_t^{\rm ASEP}$ at time $t = 1000$ (resp. $t = 5000$). The number of trials is $10^6$.
		}\label{asepright}
	\end{figure}
	
The convergence to $\hat u$ is confirmed at the level of the first three standard statistics. Indeed, denote by $\overline{\E}(B_{t_i}^{{\rm ASEP}})$ the empirical mean of $B_{t_i}^{{\rm ASEP}}$, and
	\begin{equation}\label{e36}
		\widehat{\E}\left(B_{t_i}^{{\rm ASEP}}\right) = \left|\overline{\E}(B_{t_i}^{{\rm ASEP}}) - \E\left(\hat u\right)\right|,
	\end{equation}
and similarly for other statistics: ${\rm Var}$, ${\rm Skew}$, and ${\rm Kurt}$.

\begin{figure}[h]
		\centering
		\includegraphics[height=4cm]{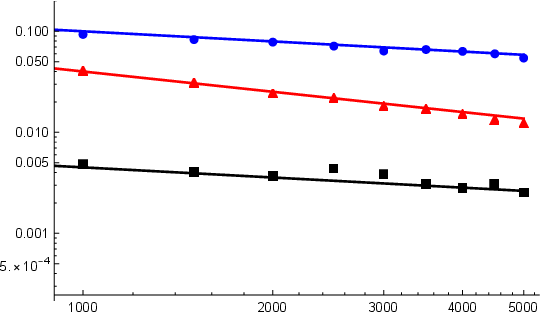}
		\caption{Log-log plot for the statistics of $B_t^{\mathrm{ASEP}}$. The square, triangle, and circle markers represent the points corresponding to $(t_i, \widehat{\mathbb{E}}(B_t^{\mathrm{ASEP}}))$, $(t_i, \widehat{\mathrm{Var}}(B_t^{\mathrm{ASEP}}))$, and $(t_i, \widehat{\mathrm{Skew}}(B_t^{\mathrm{ASEP}}))$, respectively, where $t_i \in \{1000, 1500, \ldots, 5000\}$. The blue, red, and black solid lines are the reference lines $(t, t^{-1/3})$, $(t, 4t^{-2/3})$, and $(t, 0.045t^{-1/3})$, respectively. The number of trials is $10^6$.
		}\label{asependpointstatisticc}
	\end{figure}
In Figure~\ref{asependpointstatisticc} we have a log-log plot of these statistics and we clearly see that in all cases the convergence to the ones of $\hat u$ are power-law. More precisely, $\widehat{\E}\left(B_{t_i}^{{\rm ASEP}}\right)$ and $\widehat{\mathrm{Skew}}(B_t^{\mathrm{ASEP}})$ converges to $0$ as $t^{-1/3}$, while the speed of convergence $\widehat{\mathrm{Var}}(B_t^{\mathrm{ASEP}})$ is $t^{-2/3}$. As often is the case, the statistics of the kurtosis is much more noisy and it did not provide reliable numbers, so we did not include in the plot.

\subsubsection{On the discrepancy $D_N(t)$}\label{asepdis}
	
Next, we provide numerical evidence that the discrepancy $D_n(t)$ tends to a random variable without scaling in time. In the left panel of Figure~\ref{diffPdf}, we illustrate the empirical density function of $D_n(t)$ at $t = 200$, $t = 1000$, and $t = 2000$, while in the right panel of Figure~\ref{diffPdf} we show the log-linear plots of it. As time increases, one observes that the probability of the very small values of $D_N(t)$ slightly decreases, but more importantly, the probabilities of the large values considerably decrease, indicating that $D_N(t)$ tends to a random variable.	
		\begin{figure}[h]
		\centering
		\begin{subfigure}
			\centering
			\includegraphics[height=4cm]{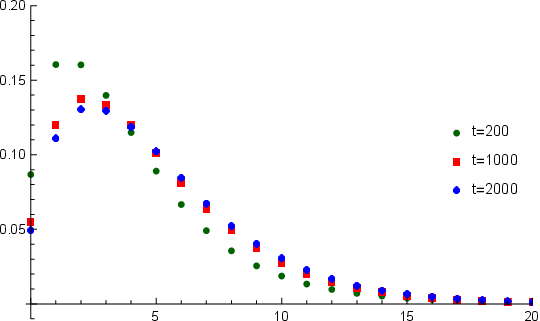}
		\end{subfigure}
		\begin{subfigure}
			\centering
			\includegraphics[height=4cm]{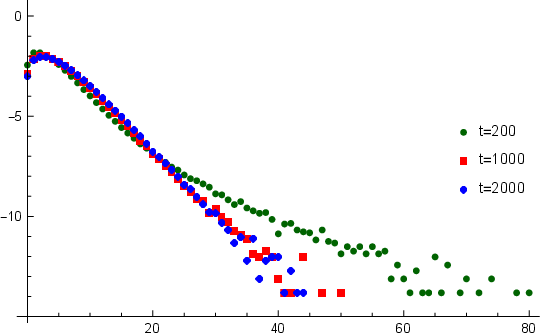}
		\end{subfigure}
		\caption{Probability density function $f_{D_n(t)}$ in ASEP with $p = 3/4$.			
			Left panel: Original density function $(x, f_{D_n(t)}(x))$.			
			Right panel: Log-linear plot $(x, \log f_{D_n(t)}(x))$. The green, red, and blue points in the figure correspond to $t = 200$, $t = 1000$, and $t = 2000$, respectively. The number of trials is $10^6$.	
		 }\label{diffPdf}
	\end{figure}

We also present the statistics of $D_n(t)$ at times $t = 400, 800, \dots, 2000$ in Table~\ref{statistic}. The expectation and variance slightly increase with time, and the skewness decreases over time. However, the rates of change are extremely slow, and The rate of change is gradually slowing down over time. As for the kurtosis, compared to the other statistics, it appears to exhibit greater numerical instability.
	
	\begin{table}[h]
		\centering
		\begin{tabularx}{0.779\textwidth} {|c|c|c|c|c|  }
			\hline
			$t$ & $\E\left(D_N(t)\right)$ & $\operatorname{Var}\left(D_N(t)\right)$ &$\operatorname{Skew}\left(D_N(t)\right)$ &$\operatorname{Kurt}\left(D_N(t)\right)$ \\
			\hline
			$400$ &4.34   & 12.96  &1.80 &9.73\\
			$800$ &4.62   & 13.10  &1.47 &6.79\\
			$1200$ &4.79   & 13.57  &1.41 &6.33\\
			$1600$ &4.88   & 13.81  &1.38 &6.15\\
			$2000$ &4.95   & 14.07  &1.34 &5.88\\
			\hline
		\end{tabularx}
		\caption{Empirical mean, variance, skewness, and kurtosis of the discrepancy $D_n(t)$ with $t \in \{400, 800, \ldots, 2000\}$. The number of trials is $10^6$.
			 }\label{statistic}
	\end{table}
	
	\subsection{Results for speed changed ASEP}
	\subsubsection{Fluctuations of tagged particle position}\label{asepscpp}
    We made the simulation of the \newasep\ with jump rate given by \eqref{asepnjumprate} with parameters $\beta=\log 4$ and $E=\infty$ (particles can only jump to right). Denote $X_{{\rm GOE}}$ as a random variable such that
	\begin{equation}
		\Pb\left(X_{{\rm GOE}}\leq s\right)=F_{{\rm GOE}}(2s).
	\end{equation}
		In Figure~\ref{ne}, we compare the empirical density function of $X^{\rm resc}_N(t)$ at times $t = 200$, $t = 600$, and $t = 1000$ with the density function of $X_{{\rm GOE}}$. One clearly see that the shape of the two functions are very close, indicating that $X^{\rm resc}_N(t)$ indeed converges in distribution to $X_{{\rm GOE}}$. Also, one can see a slight shift to the right, which decreases over time. This is a general fact occurring in models in the KPZ class as discussed in Remark~\ref{fub} below.
	
	\begin{figure}[h]
		\centering
		\begin{subfigure}
			\centering
			\includegraphics[height=5cm]{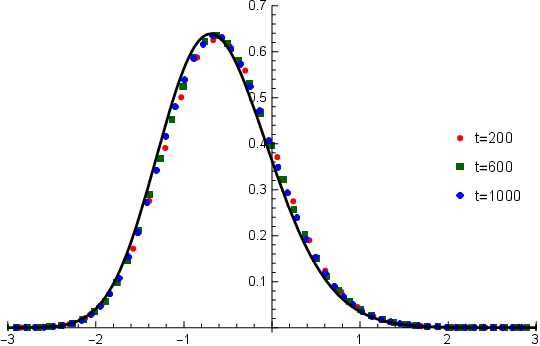}
		\end{subfigure}

		\caption{Numerical verification for Conjecture~\ref{conjp}. The red, green, and blue dots represent the density function of $X^{\rm resc}_n(t)$ with $t \in \{200, 600, 1000\}$. The black solid line is given by $(s, 2F'_{\rm GOE}(2s))$. The number of trials is $10^6$.
		}
		\label{ne}
	\end{figure}
	
We use the same notation as in \eqref{e36}, that is,
	\begin{equation}
		\widehat{\E}\left(X^{\rm resc}_N(t)\right) = \left|\overline{\E}(X^{\rm resc}_N(t)) - \E\left(X_{{\rm GOE}}\right)\right|
	\end{equation}
	and also for other statistics. To illustrate the convergence, we show in Figure~\ref{asepscoristatisticlog} the log-log plots of these statistics as a function of time. The reference line for the expectation has slope $-1/3$, while for the other three statistics the slope is $-2/3$. Thus the speed of convergence of the expectation of $X^{\rm resc}_N(t)$ is $\mathcal{O}(t^{-1/3})$, whereas for the other three statistics is $\mathcal{O}(t^{-2/3})$.
	\begin{figure}[h]
		\centering
		\begin{subfigure}
			\centering
			\includegraphics[height=4cm]{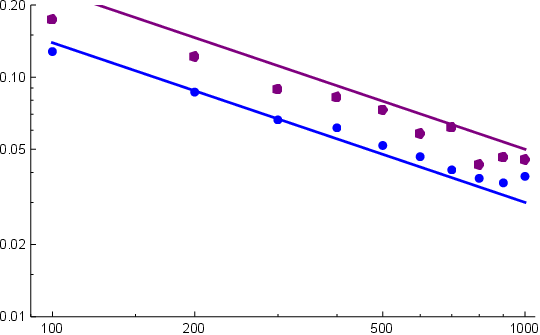}
		\end{subfigure}
		\begin{subfigure}
			\centering
			\includegraphics[height=4cm]{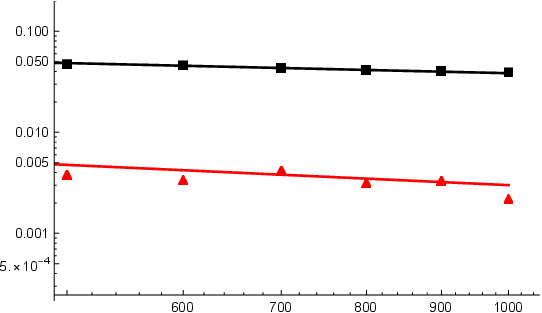}
		\end{subfigure}
		\caption{Log-log plot for the absolute difference between the statistics of $X^{\rm resc}_N(t)$ and $X_{{\rm GOE}}$.			
			Left panel: The pentagon represents $\widehat{\mathrm{Kurt}}(X^{\rm resc}_N(t))$ and the circle represents $\widehat{\mathrm{Skew}}(X^{\rm resc}_N(t))$. The purple solid line is the reference line $(t, 5t^{-2/3})$, and the blue line is $(t, 3t^{-2/3})$.			
			Right panel: The rectangle represents $\widehat{\mathrm{Var}}(X^{\rm resc}_N(t))$ and the triangle represents $\widehat{\E}(X^{\rm resc}_N(t))$. The black solid line is the reference line $(t, 5t^{-1/3}/13)$, and the red line is $(t, 3t^{-2/3}/10)$.		
			The number of trials is $10^6$.
		}\label{asepscoristatisticlog}
	\end{figure}

	\begin{rem}\label{fub}
In the unscaled variables, the mean has a shift of order one, which is typical for models in the KPT universality class, see~\cite{TS10,TSSS11} for real experimental results, and \cite{FF11} for theoretical results. In our case, since the slope of the black line in Figure~\ref{asepscoristatisticlog} is about $0.38(5)$, so we define
		\begin{equation}\label{hatx}
			\hat X^{\rm resc}_N(t) = \frac{X_N^{{\newasep}}(t) - X^{\newasep}_N(0) - 2J(\beta,E)t + 0.385}{-2\Gamma(\beta,E)^{1/3} t^{1/3}}.
		\end{equation}
Then one observes better convergence behavior since now also the expectation converges at speed $t^{-2/3}$, see Figure~\ref{finite} compared to Figure~\ref{ne}.
		\end{rem}
		\begin{figure}[h]
			\centering
			\includegraphics[height=5cm]{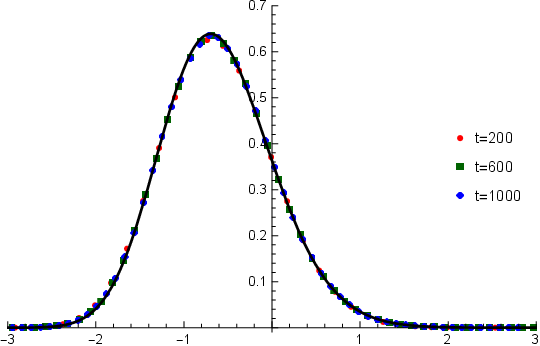}
			\caption{The red, green, and blue dots represent the density function of $\widehat X^{\rm resc}_n(t)$ with $t \in \{200, 600, 1000\}$. The black solid line is given by $(s, 2F'_{\rm GOE}(2s))$. The number of trials is $10^6$.
			}\label{finite}
		\end{figure}

	\subsubsection{End-point of backwards geodesic}\label{asepscbkp}

By the construction of the backwards geodesic on particles, even for density $\rho=1/2$, the law of the starting point is not symmetric unlike the law of $\hat u$. However as we numerically see, this asymmetry asymptotically will disappear. The asymmetry would not be present for $\rho=1/2$ if one would have considered the backwards geodesic for height functions as for example in~\cite{BF20}.

In Figure~\ref{neright}, we present a comparison of the empirical density function of $B^{{\rm ASEPsc}}_t$ defined in~\eqref{basepsc} at times $t = 200$, $t = 600$, and $t = 1000$ with the density function of $\hat u$. From the probability density plot, we observe that as time increases, the density function of $B^{{\rm ASEPsc}}_t$ slowly approaches that of $\hat u$. As the empirical mean is converging quite slowly (see below) we also plot the density function of the centered version of $B^{{\rm ASEPsc}}_t$, that is,
	\begin{equation}
		\hat B^{{\rm ASEPsc}}_t = B^{{\rm ASEPsc}}_t - \overline{\E}\left(B^{{\rm ASEPsc}}_t\right).
	\end{equation}
In the right panel of Figure~\ref{neright} one better sees the convergence trend over time.
			\begin{figure}[h]
		\centering
		\begin{subfigure}
			\centering
			\includegraphics[width=0.45\textwidth]{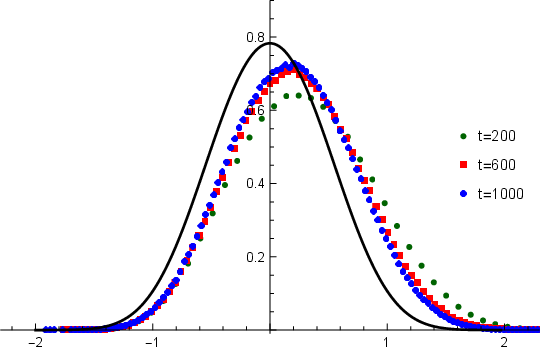}
		\end{subfigure}
		\begin{subfigure}
			\centering
			\includegraphics[width=0.45\textwidth]{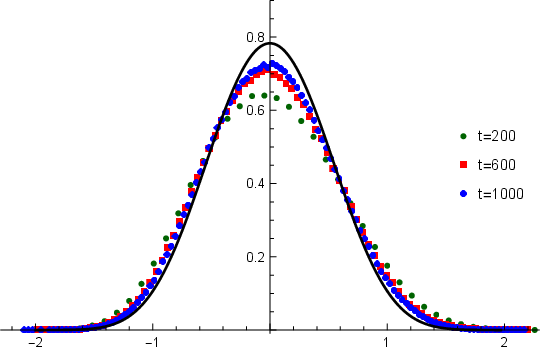}
		\end{subfigure}
		\caption{Numerics for Conjecture~\ref{conj}. 			
			Left panel: The green, red, and blue dots represent the density function of $B_t^{{\rm ASEPsc}}$ with $t \in \{200, 600, 1000\}$.			
			Right panel: The green, red, and blue dots represent the density function of $\hat B_t^{{\rm ASEPsc}}$ with $t \in \{200, 600, 1000\}$. The black solid line is the density function of $\hat u$.			
			The number of trials is $10^6$.
		}\label{neright}
	\end{figure}

For \newasep\ the convergence is slower than for ASEP and thus it is even more important to look at the convergence of the statistics. Recall the notations as in \eqref{e36},
$		\widehat{\E}\left(B^{{\rm ASEPsc}}_t\right) = \left|\overline{\E}(B^{{\rm ASEPsc}}_t) - \E\left(\hat u\right)\right|$ and similarly for the other statistics. In Figure~\ref{AsepscEndpointStatisticloga}, we plot the log-log plots of the first four statistics as compared with the ones of $\hat u$. We see that the expectation of $B^{{\rm ASEPsc}}_t$ converges as $t^{-1/3}$, while the other three statistics converges faster, namely as $t^{-2/3}$.

		\begin{figure}[h]
		\centering
		\includegraphics[height=4.5cm]{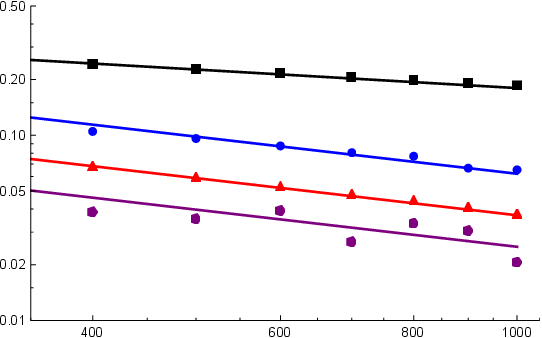}
		\caption{Log-log plot for the statistic of $B^{{\rm ASEPsc}}_t$. The rectangle, triangle, circle, and pentagon represent the numerical values of $\widehat{\E}(B^{{\rm ASEPsc}}_t)$, $\widehat{\mathrm{Var}}(B^{{\rm ASEPsc}}_t)$, $\widehat{\mathrm{Skew}}(B^{{\rm ASEPsc}}_t)$, and $\widehat{\mathrm{Kurt}}(B^{{\rm ASEPsc}}_t)$, respectively. The black, blue, red, and purple solid lines are the reference lines $(t, 1.8 t^{-1/3})$, $(t, 6.2 t^{-2/3})$, $(t, 3.7 t^{-2/3})$, and $(t, 2.5 t^{-2/3})$, respectively. The number of trials is $10^6$.
		}\label{AsepscEndpointStatisticloga}
	\end{figure}
	
	\appendix
	\section{Numerical implementation}
	\subsection{Implementation for $\argmax_{u\in\R} \{{\cal A}_2(u)-u^2\}$}
	The implementation of distribution of $\hat u=\argmax \{{\cal A}_2(u)-u^2\}$ is based on the formula in\footnote{Another formula is obtained in \cite{SCH12}, and it was shown in \cite{BKS12} that they are the same.} \cite{MQR13}. Let ${\rm Ai}$ be Airy function and operator $P_0$ as the projection on positive real line, i.e., $P_0f(x)=\Id_{x>0}f(x)$. For $t,m\in\R$, define the functions 	\begin{equation}
		\begin{aligned}
			B_m(x,y)&=\mathrm{Ai}(x+y+m),\\
			\psi_{t, m}(x)&=2 e^{x t}\left[t \Ai\left(x+m+t^2\right)+\Ai^{\prime}\left(x+m+t^2\right)\right]
		\end{aligned}
	\end{equation}	
	and the kernel
	\begin{equation}
		\Psi_{t, m}(x, y)=2^{1 / 3} \psi_{t, m}\left(2^{1 / 3} x\right) \psi_{-t, m}\left(2^{1 / 3} y\right).
	\end{equation}		
	Define also $\mathcal M=\max_{u\in\R}\{{\cal A}_2(u)-u^2\}$, the joint density of $(\hat u,\mathcal M)$ is~\cite[Theorem~2]{MQR13}:
	\begin{equation}\label{jointdistribution}
		\begin{aligned}
			f(t, m) =\det\left(\Id-P_0 B_{4^{1 / 3}} P_0+P_0 \Psi_{t, m} P_0\right)-F_{\mathrm{GOE}}\left(4^{1 / 3} m\right),
		\end{aligned}
	\end{equation}
	where $F_{\mathrm{GOE}}$ is GOE Tracy-Widom distribution. Both quantities on the right hand side of \eqref{jointdistribution} can be obtained numerically via Bornemann's method \cite{Born08}. Integrating over $\mathcal M$, one obtains the distribution of $\hat u$.
	\subsection{Implementation for ASEP}
	 Denote the total simulation time as $t_{\max }$ and the number of particles as $N_{\max}$. In the implementation we set $t_{\max}=5000$ and $N_{\max}=3750$ and the index of particle from where we construct backwards geodesic is given by $N=937$. The total number of trials is $10^6$. For a fixed realization of $X(t_{\max})$, we will record $N(t_i\downarrow0)$ with $t_i\in\{500,1000,\ldots,5000\}$.
	 \subsubsection{For original process}
	 The initial condition is\footnote{In the simulation we used a left-to-right ordering unlike formulas in the theoretical part of this paper.} $X_n(0)=2n-1$ for $n\in\{1,\ldots,N\}$. We will first prepare three random matrices $\mathcal J,\mathcal U,\mathcal D\in\mathcal R_{\geq 0}^{N\times C_Nt_{\max}}$ for some $C_N>0$ depending on $N$ in the following way:
	\begin{enumerate}
		\item $\mathcal J_{i,j}\sim {\rm Exp}(1)$ are i.i.d.\ random variables. Column $n$ of this matrix provides the information when particle $X_n$ will attempt to jump: the $m-$the attempted jump time of particle $X_n$ is $\sum_{i=1}^{m}\mathcal J_{i,n}$.
		\item $\mathcal U_{i,j}\sim {\rm Unif}([0,1])$ are i.i.d.\ random variables and
		\begin{equation}
			\mathcal D_{i,j}=\begin{cases}
				1,\quad&\mathrm{if}\ \mathcal U_{i,j}\leq p,\\
				-1,&\mathrm{otherwise}.
			\end{cases}
		\end{equation}
		The  $n-$th column of $\mathcal D$ tells us in which direction particle $X_{n}$ should attempt to jump: if $\mathcal D_{m,n}=1$, then the dirction of $m-$th attempted jump of particle $X_n$ is from left to right otherwise from right to left.
	\end{enumerate}
	In the simulation, we will choose $C_N$ (depending on $N$) large enough such that $\min_n\{\sum_{i=1}^{C_Nt_{\max}}\mathcal J_{i,n}\}\geq t_{\max}$ with high probability. For a given initial condition $X(0)$ and matrices $\mathcal J,\mathcal D$, the final state $X(t)$ is a deterministic result. We record current particle's time-space position in two vectors $\mathcal S,\mathcal T\in\mathcal R_{\geq 0}^{N+1}$ with $\mathcal T_1=\mathcal T_{N+1}=t_{\max}$ and $\mathcal S_{N+1}=\infty$. The information about next attempted jump direction is stored in vector ${\cal A}\in\{-1,1\}^{N}$.
	The updating rule is as follows:
	\begin{enumerate}
		\item \textbf{Inital step:} Set $\mathcal S_{n}=X_n(0)$, $\mathcal T_{n}=\mathcal J_{1,n}$, ${\cal A}_n=\mathcal D_{1,n}$ and $\mathcal H_n=2$ for all $n\in\{1,\ldots, N\}$.
		\item \textbf{Induction step:} For given $\mathcal S,\ \mathcal T,\ {\cal A}$ and $\mathcal H$, until $\min\{\mathcal T_n\}\geq t_{\max}$, we choose $m\in\{2,\ldots,N\}$ such that $\mathcal T_{m}\leq \min\{\mathcal T_{m+1},\mathcal T_{m-1}\}$ and update
		\begin{equation}
			\mathcal S_{m}=\begin{cases}
				\mathcal S_{m}+1,\quad&\mathrm{if}\ {\cal A}_{m}=1\ \mathrm{and}\ \mathcal S_{m+1}\neq \mathcal S_{m}+1,\\
				\mathcal S_{m}-1,\quad&\mathrm{if}\ \mathcal D_{m}=-1\ \mathrm{and}\ \mathcal S_{m-1}\neq \mathcal S_{m}-1,\\
				\mathcal S_{m},&\mathrm{otherwise.}
			\end{cases}
		\end{equation}
		We also update its current time position as $\mathcal T_m=\mathcal T_m+\mathcal J_{\mathcal H_m,m}$. In the case when ${\cal A}_{m}=1$ and $\mathcal S_{m+1}= \mathcal S_{m}+1$ (resp.\ ${\cal A}_{m}=-1$ and $\mathcal S_{m-1}= \mathcal S_{m}-1$), this jump will be recorded as suppressed left to right jump (resp.\ right to left jump). Updated ${\cal A}_m=\mathcal D_{\mathcal H_m,m}$ and $\mathcal H_m=\mathcal H_m+1$.
	\end{enumerate}
	\subsubsection{For backwards geodesic}
	
	Along the construction of ASEP, the information about the suppressed left to right (resp.\ right to left) jumps will be stored in matrix $\mathcal T^{\not\rightarrow}$ (resp.\ $\mathcal T^{\not\leftarrow}$): $\mathcal T_{i,n}^{\not\rightarrow}$ is the time of $i-$th suppressed left to right jump of particle $X_n$. For a given time $t\in[0,t_{\max}]$ and $n\in\{1,\ldots, N\}$, the end-point of backwards geodesic of particle $X_n$ is obtained as follows:
	\begin{enumerate}
		\item \textbf{Initial step:} At time $t$, we set $N(t\downarrow t)=n$.
		\item \textbf{Induction step:} Given $m=N(t\downarrow s)$, find $\tau=\max\{r\leq s| r\in \mathcal T^{\not\rightarrow}_{m}\cup\mathcal T^{\not\leftarrow}_m\}$. If $\tau\in\mathcal T_m^{\not\rightarrow}$ (resp.\ $\tau\in \mathcal T^{\not\leftarrow}$), then update $N(t\downarrow\tau)=m+1$ (resp.\ $N(t\downarrow\tau)=m-1$).
	\end{enumerate}

	\subsubsection{For discrepancy in ASEP}\label{discrepancyasep}
	For a given $t\in[0,t_{\max}]$ and $n\in\{1,\ldots,N\}$, after obtaining $N(t\downarrow 0)$ described as above, we will implement $D_n(t)$ defined in \eqref{dnt1} as follows:
	\begin{enumerate}
		\item Consider a new ASEP $\hat X(t)$ with initial condition given by $\hat X_j(0)=X_{N(t\downarrow 0)}(0)-(N(t\downarrow 0)-j)$ for all $j\leq N(t\downarrow 0)$.
		\item The time evolution follows the same realization of random matrices $\mathcal J$ and $\mathcal D$ as the one for original process $X(t)$: for $n\leq N(t\downarrow 0)$, $\mathcal J_{i,j}$ (resp.\ $\mathcal D_{i,j}$) is the $i-$th attempted jump time  (resp.\ the direction of $i-$th attempted jump) of particle $\hat X_j$.
	\end{enumerate}
	Then $D_n(t)$ in \eqref{dnt1} is given by $\hat X_n(t)-X_n(t)$. Due to the high time consuming of the simulation, here we choose $t_{\max}=2000$ and $N_{\max}=1500$. The total number of trials is $10^6$. The index of particle from where we constructed backwards geodesic is given $N=375$. For each realization of $X(t)$, we  construct $\hat X(t_i)$ with $t_i\in\{200,400,\ldots,2000\}$.

	\subsection{Implementation for speed changed}\label{resultasepsc}
	
	Due to the absence of homogeneity of jump rate, the implementation of \newasep\ is slightly different from the one for ASEP, namely, it does not make sense to prepare matrices $\mathcal J,\mathcal D$ once for all time as we did for ASEP. The jump rates illustrated in Figure~\ref{jumprate} do not include the situations of suppressed jumps. In the same spirit of ASEP, also for \newasep\ we assign jump trials from $j$ to $j+1$ (resp.\ from $j+1$ to $j$) as in Figure~\ref{jumprate} without at first caring whether the arrival site is occupied or empty. If the arrival site for right jumps at $j+1$ (resp.\ left jumps at $j$) is occupied, the it will be a suppressed jump.

	For given jump rate, we then determine $X(t)$ as follows:
	\begin{enumerate}
		\item \textbf{Initial step:}
		\begin{enumerate}
			\item For space-time trajectory: set $\mathcal S_n=2n-1$ and $\mathcal T_n=0$ for all $n$.
			\item For jump schema: For each $n\in\{1,\ldots, N\},$ we find its left (resp.\ right) jump rate $\theta^n_\ell$ (resp.\ $\theta^n_r$) according to its local environment.
			Define $\mathcal R_n=\theta_\ell^n+\theta_r^n$.
			\item For probability being chosen: Let $\mathcal P_n=(\sum_{i=1}^n\mathcal R_i)/(\sum_{i=1}^N\mathcal R_i)$ for all $n\geq 1$ and $\mathcal P_0=0$.
		\end{enumerate}
		\item \textbf{Induction step:} For given $\mathcal S,\ \mathcal T,\ \mathcal R$ and $\mathcal P$, until $\min_n\{\mathcal T_n\}\geq t_{\max}$, we let $\mathcal U\sim {\rm Unif}([0,1])$ and find $m\in\{1,\ldots,N\}$ such that $\mathcal P_{m-1}<\mathcal U\leq \mathcal P_m$. Then we update the $\mathcal S,\ \mathcal T,\ \mathcal R$ and $\mathcal P$ as follows:
		\begin{enumerate}
			\item \textbf{For time position:} Set $\mathcal T_n=\mathcal T_n+\mathcal J$, where $\mathcal J\sim {\rm Exp}(\mathcal R_m)$ for all $n$.
			\item \textbf{For space position:} Let $\mathcal K\sim Unif([0,1])$ be i.i.d.\ and set
			\begin{equation}
				\mathcal S_m=\begin{cases}
					\mathcal S_m+1,\quad&\mathrm{if}\ \mathcal K_m\leq \frac{\theta_r^m}{\theta_\ell^m+\theta_r^m}\ \mathrm{and}\ \mathcal S_{n+1}>\mathcal S_m+1,\\
					\mathcal S_m-1,&\mathrm{if}\ \mathcal K_m> \frac{\theta_r^m}{\theta_\ell^m+\theta_r^m}\ \mathrm{and}\ \mathcal S_{n}>\mathcal S_{n-1}+1,\\
					\mathcal S_m,&\mathrm{otherwise.}
				\end{cases}
			\end{equation}
			In the case when $\mathcal K_m\leq \frac{\theta^m_r}{\theta^m_\ell+\theta^m_r}$ and $\mathcal S_{n+1}=\mathcal S_m+1$ (resp.\ $\mathcal K_m> \frac{\theta_r^m}{\theta_\ell^m+\theta_r^m}$ and $\mathcal S_{n}=\mathcal S_{n-1}+1$), we mark this attempted jump as suppressed right (resp.\ left) jump.
			\item Updated $\mathcal P$ and $\mathcal R$ according to current status of $\mathcal S$.
		\end{enumerate}
	\end{enumerate}
	After constructing the process $X(t)$, the end-point of backwards geodesic can be constructed as same as the one for ASEP. Due to the extra time needed to determine the jump rate, we choose $t_{\max}=1000$ and $N_{\max}=1000$,  The index of target particle, from where we construct backwards geodesic, is given as $N=187$.
	

\end{document}